\documentclass[12pt]{amsart}
\usepackage[utf8]{inputenc}
\usepackage{setspace}
\usepackage{thm-restate}
\usepackage[shortlabels]{enumitem}
\usepackage[T1]{fontenc}
\usepackage{tikz}
\usetikzlibrary{patterns}
\usetikzlibrary{decorations.pathreplacing}
\usepackage{tkz-graph}

\usetikzlibrary{arrows}
\usetikzlibrary{decorations.markings}
\usepackage{verbatim}
\usepackage[margin=1in]{geometry}
\usepackage{amsmath,amsthm,amssymb,fancyhdr,tikz-cd}
\usepackage{mathrsfs}
\usepackage{hyperref}
\hypersetup{
    colorlinks = true,
    citecolor = orange,%
}
\usepackage{cleveref}
\usepackage{ytableau}
\usepackage{subfig}
\renewcommand{\phi}{\varphi}
\newcommand{\tableau}[1]{\begin{ytableau}#1\end{ytableau}}
\newcommand{\tld}[1]{\widetilde{#1}}
\renewcommand{\bar}[1]{\overline{#1}}

\theoremstyle{plain}
\newtheorem{theorem}{Theorem}[section]
\newtheorem{restate}{Theorem}[section]
\newtheorem*{MT}{Main Theorem}
\newtheorem{lemma}[theorem]{Lemma}
\newtheorem{prop}[theorem]{Proposition}

\newtheorem{corollary}[theorem]{Corollary}

\theoremstyle{definition}
\newtheorem{remark}[theorem]{Remark}
\newtheorem{example}[theorem]{Example}
\newtheorem{definition}[theorem]{Definition}
\newtheorem{conjecture}[theorem]{Conjecture}

\DeclareMathOperator{\ddeg}{ddeg}
\DeclareMathOperator{\rank}{rank}

\DeclareMathOperator{\row}{Row}
\DeclareMathOperator{\swap}{swap}

\DeclareMathOperator{\IT}{IT}
\DeclareMathOperator{\SIT}{SIT}
\DeclareMathOperator{\PP}{PP}
\DeclareMathOperator{\kjdt}{\textit{K}-jdt}
\DeclareMathOperator{\kpro}{\textit{K}-pro}
\DeclareMathOperator{\kbk}{\textit{K}-BK}
\usepackage{etoolbox}

\usepackage{style}

\def\SetFancyGraph {
	\SetVertexMath
	\GraphInit[vstyle=Art]
	\SetUpVertex[MinSize=2pt]
	\SetVertexLabel
	\tikzset{VertexStyle/.style = {shape = circle,shading = ball,ball color = black,inner sep = 1.5pt}}
	\SetUpEdge[color=black]
	\tikzset{->-/.style={decoration={ markings, mark=at position 0.8 with {\arrow{>}}},postaction={decorate}}}
	\tikzset{->--/.style={decoration={ markings, mark=at position 0.55 with {\arrow{>}}},postaction={decorate}}}
}

\makeatletter
\patchcmd{\@setauthors}{\MakeUppercase}{}{}{}
\makeatother

\AtBeginDocument{\addtocontents{toc}{\protect\setlength{\parskip}{0.2pt}}}

\begin{document}

\newcommand\todo[1]{\textcolor{red}{TODO: #1}}
\newcommand{\dpg}{doppelg\"anger}
\newcommand{\Dpg}{Doppelg\"anger}
\newcommand{\PL}{PL}
\newcommand{\mb}[1]{\mathbb{#1}}
\newcommand{\mc}[1]{\mathcal{#1}}
\newcommand{\Image}{\text{Im }}
\newcommand{\Gr}{\text{Gr}}
\newcommand{\Mat}{\text{Mat}}
\newcommand{\Cr}{\text{cr}}
\newcommand{\Le}{\scalebox{-1}[1]{L}}
\newcommand{\tlN}[1]{\tld{\mc{N}}(}
\newcommand{\jdt}{{jeu-de-taquin}}
\newcommand{\julian}[1]{\textcolor{green}{\sffamily ((JULIAN: #1))}}
\newcommand{\quang}[1]{\textcolor{cyan}{\sffamily ((QUANG: #1))}}
\newcommand{\calvin}[1]{\textcolor{orange}{\sffamily ((CALVIN: #1))}}
\newcommand{\syl}[1]{\textcolor{red}{\sffamily (SYLVESTER: #1)}}

\newcommand{\calP}{\mathcal{P}}

\title{Rowmotion Orbits of Trapezoid Posets}

\author[Q.~Dao]{Quang Vu Dao{$^\clubsuit$}}

\email{{$^\clubsuit$}qvd2000@columbia.edu}
\author[J.~Wellman]{Julian Wellman{$^\diamondsuit$}} 

\email{{$^\diamondsuit$}wellman@mit.edu}
\author[C.~Yost-Wolff]{Calvin Yost-Wolff{$^\heartsuit$}}

\email{{$^\heartsuit$}calvinyw@mit.edu}
\author[S.~Zhang]{Sylvester W.~Zhang{$^\spadesuit$}}
\email{{$^\spadesuit$}zhan5102@umn.edu}
\address{{$^\clubsuit$} Department of Mathematics, University of Columbia, 2990 Broadway, New York, NY 10027}
\address{{$^\diamondsuit$}{$^\heartsuit$} Department of Mathematics, Massachusetts Institute of Technology, 182 Memorial Dr., Cambridge, MA 02142}
\address{{$^\spadesuit$} School of Mathematics, University of Minnesota -- Twin Cities, 206 Church St SE, Minneapolis, MN 55455}

\thanks{{$^\clubsuit$}{$^\diamondsuit$}{$^\heartsuit$}{$^\spadesuit$}Supported by NSF RTG grant DMS-1148634 at the 2019 combinatorics REU program at the School of Mathematics of the University of Minnesota, Twin Cities.}

\maketitle
\begin{abstract}
Rowmotion is an invertible operator on the order ideals of a poset which has been extensively studied and is well understood for the rectangle poset.
In this paper, we show that rowmotion is equivariant with respect to a bijection of Hamaker, Patrias, Pechenik and Williams between order ideals of rectangle and trapezoid posets, thereby affirming a conjecture of Hopkins that the rectangle and trapezoid posets have the same rowmotion orbit structures. Our main tools in proving this are $K$-jeu-de-taquin and (weak) $K$-Knuth equivalence of increasing tableaux. We define \emph{almost minimal tableaux} as a family of tableaux naturally arising from order ideals and show for any $\lambda$, the almost minimal tableaux of shape $\lambda$ are in different (weak) $K$-Knuth equivalence classes. We also discuss and make some progress on related conjectures of Hopkins on down-degree homomesy. 
\end{abstract}
\setcounter{tocdepth}{1}

\tableofcontents
\vspace{-1.5cm}
\section{Introduction}
Rowmotion, denoted $\row$, is an invertible operator on the order ideals of any partially ordered set. For an order ideal $I$, $\row(I)$ is the order ideal generated by the minimal elements of the complement of $I$.  Rowmotion was first introduced by Duchet \cite{rowfirst} and has been extensively studied by many different authors (including Brouwer-Schrijver \cite{bs74}, Fon-der-Flaass \cite{fon1993orbits}, Cameron-Fon-der-Flaass \cite{toggles}, Panyushev \cite{Pan09}, and Striker-Williams \cite{prorow}). The name `rowmotion' is due to Striker and Williams \cite{prorow}. For more history on rowmotion, see \cite[Section 7.1]{TW09}.

We are interested in the action of rowmotion on the following two particular posets: \begin{itemize}[noitemsep]
    \item the \emph{rectangle poset} $\mathscr{R}(a,b):=\{(i,j)\in\mathbb Z^2:1\le i\le a,1\le j\le b\}$, and
    \item the \emph{trapezoid poset} $\mathscr{T}(a,b):=\{(i,j)\in\mathbb Z^2:1\le i\le a,i\le j\le a+b-i\}$
\end{itemize}
for some fixed $a\le b$, where the partial order is induced from the natural order on $\mathbb Z^2$. %
Figure~\ref{fig:rect_trap_normal} gives an example of the Hasse diagrams of two posets $\mathscr R(3,5)$ and $\mathscr T(3,5)$.

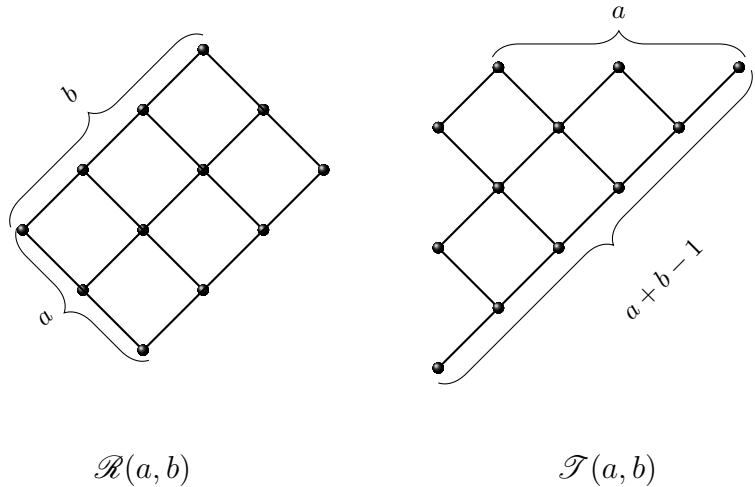
\begin{figure}[h]
\begin{center}
\begin{tikzpicture}[xscale=-0.8,yscale=0.8]
	\SetFancyGraph
	\Vertex[NoLabel,x=0,y=0]{1}
	\Vertex[NoLabel,x=-1,y=1]{2}
	\Vertex[NoLabel,x=1,y=1]{3}
	\Vertex[NoLabel,x=-2,y=2]{4}
	\Vertex[NoLabel,x=0,y=2]{5}
	\Vertex[NoLabel,x=2,y=2]{6}
	\Vertex[NoLabel,x=-3,y=3]{7}
	\Vertex[NoLabel,x=-1,y=3]{8}
	\Vertex[NoLabel,x=1,y=3]{9}
	\Vertex[NoLabel,x=-2,y=4]{10}
	\Vertex[NoLabel,x=0,y=4]{11}
	\Vertex[NoLabel,x=-1,y=5]{12}
	\Edges[style={thick}](1,7)
	\Edges[style={thick}](3,10)
	\Edges[style={thick}](6,12)
	\Edges[style={thick}](1,6)
	\Edges[style={thick}](2,9)
	\Edges[style={thick}](4,11)
	\Edges[style={thick}](7,12)
	\draw [decorate,decoration={brace,amplitude=10pt},yshift=7pt] (2.2,1.8)--(-1,5) node [black,midway,yshift=0.5cm,xshift=-0.5cm]  {\rotatebox{45}{\footnotesize $b$}};

	\draw [decorate,decoration={brace,amplitude=10pt,mirror},yshift=-2pt] (2.1,2.1)--(-0.1,-0.1)  node [black,midway,xshift=-0.5cm,yshift=-0.32cm]   {\rotatebox{45}{\footnotesize $a$}};
	
	\node (a) at (0,-2) {$\mathscr{R}(a,b)$};
\end{tikzpicture}
\quad\quad\quad
\begin{tikzpicture}[scale=0.8]
	\SetFancyGraph
	\begin{scope}[yscale=-1,xscale=1]
	\Vertex[NoLabel,x=0,y=0]{1}
	\Vertex[NoLabel,x=2,y=0]{2}
	\Vertex[NoLabel,x=4,y=0]{3}
	\Vertex[NoLabel,x=-1,y=1]{4}
	\Vertex[NoLabel,x=1,y=1]{5}
	\Vertex[NoLabel,x=3,y=1]{6}
	\Vertex[NoLabel,x=0,y=2]{7}
	\Vertex[NoLabel,x=2,y=2]{8}
	\Vertex[NoLabel,x=-1,y=3]{9}
	\Vertex[NoLabel,x=1,y=3]{10}
	\Vertex[NoLabel,x=0,y=4]{11}
	\Vertex[NoLabel,x=-1,y=5]{12}
	\Edges[style={thick}](1,4)
	\Edges[style={thick}](1,5)
	\Edges[style={thick}](2,5)
	\Edges[style={thick}](2,6)
	\Edges[style={thick}](3,6)
	\Edges[style={thick}](4,7)
	\Edges[style={thick}](5,7)
	\Edges[style={thick}](5,8)
	\Edges[style={thick}](6,8)
	\Edges[style={thick}](7,9)
	\Edges[style={thick}](7,10)
	\Edges[style={thick}](8,10)
	\Edges[style={thick}](9,11)
	\Edges[style={thick}](10,11)
	\Edges[style={thick}](11,12)
	\draw [decorate,decoration={brace,amplitude=10pt},yshift=7pt]   (4.2,-0.2)--(-1,5) node [black,midway,yshift=-0.7cm,xshift=0.9cm]  {\rotatebox{45}{\footnotesize $a+b-1$}};
	
	\draw [decorate,decoration={brace,amplitude=10pt,mirror},yshift=-2pt] (4.1,-0.1)--(-0.1,-0.1) node [black,midway,yshift=0.6cm]  {\footnotesize $a$};
	\end{scope}
	\node (a) at (1.8,-6.7) {$\mathscr{T}(a,b)$};
\end{tikzpicture}
\end{center}

    \caption{The rectangle poset and trapezoid poset}
    \label{fig:rect_trap_normal}
\end{figure}


The action of rowmotion on the rectangle poset $\mathscr R(a,b)$ is well studied and well understood. For instance, its orbit structure is completely understood: Propp and Roby (expanding upon a remark of Stanley \cite{sta09} and with further input from Hugh Thomas) explained that the action of rowmotion on the rectangle is the same as the action of cyclic rotation on binary words with $a$ $0$’s and $b$ $1$’s \cite[Proposition 26]{propp2015homomesy}. Binary words under rotation are a fundamental example for the \emph{cyclic sieving phenomenon} (see \cite{rsw04}). Propp and Roby called the correspondence between order ideals of the rectangle and binary words
the “Stanley-Thomas word” correspondence, and they used the Stanley-Thomas word to deduce various other nice properties of rowmotion on the rectangle, such as \emph{homomesy} (see Section \ref{subsec:ddeg_statistic} for more on homomesy).


The rectangle is the prototypical example of a \emph{minuscule poset}. The aforementioned results concerning the nice behavior of rowmotion on the rectangle have been extended to all minuscule posets in the work of Rush and his co-authors \cite{rs13cyclic} \cite{rush15}.

On the other hand, rowmotion on the trapezoid poset has remained mysterious. In fact, the order of rowmotion on the trapezoid was still unknown before our work. Recently, however, Sam Hopkins \cite{hopkins2019minuscule}, building on work of Hamaker-Patrias-Pechenik-Williams \cite{HPPW18} and others, made a series of conjectures describing ways in which the two posets $\mathscr R(a,b)$ and $\mathscr T(a,b)$ are remarkably similar.
In particular, Hopkins conjectured the following, which we prove as our main result.

\begin{MT}[cf. {\cite[Conjecture 4.9.1]{hopkins2019minuscule}}]\label{bigboii}
 The action of rowmotion on order ideals of the trapezoid poset $\mathscr T(a,b)$ has the same orbit structure as rowmotion on order ideals of the rectangle poset $\mathscr R(a,b)$.
\end{MT}

In 1983, Proctor \cite{pro83} proved that $\mathscr R(a,b)$ and $\mathscr T(a,b)$ have the same \emph{order polynomial}, which implies in particular that they have the same number of order ideals. Since then, many different bijections between the set of order ideals of $\mathscr R(a,b)$ and $\mathscr T(a,b)$ have been discovered \cite{stembridge86,elizalde2015bijections,courtiel2018bijections, HPPW18}, among which the bijection $\varphi$ of Hamaker, Patrias, Pechenik and Williams \cite{HPPW18} is central to our proof.

\begin{theorem}
\label{thm:commute}
The bijection $\varphi$ of \cite{HPPW18} commutes with rowmotion, i.e. for any order ideal  $\mathcal I\in J(\mathscr R(a,b))$, we have
\[\varphi \circ \row(\mathcal I)=\row \circ\ \varphi(\mathcal I).\]
\end{theorem}
\begin{example} Rowmotion on $\mathscr R(2,2)$ and $\mathscr T(2,2)$ has order 4 and more specifically one orbit of size $4$ and one orbit of size $2$. The corresponding Stanley-Thomas word is given on top of each ideal.

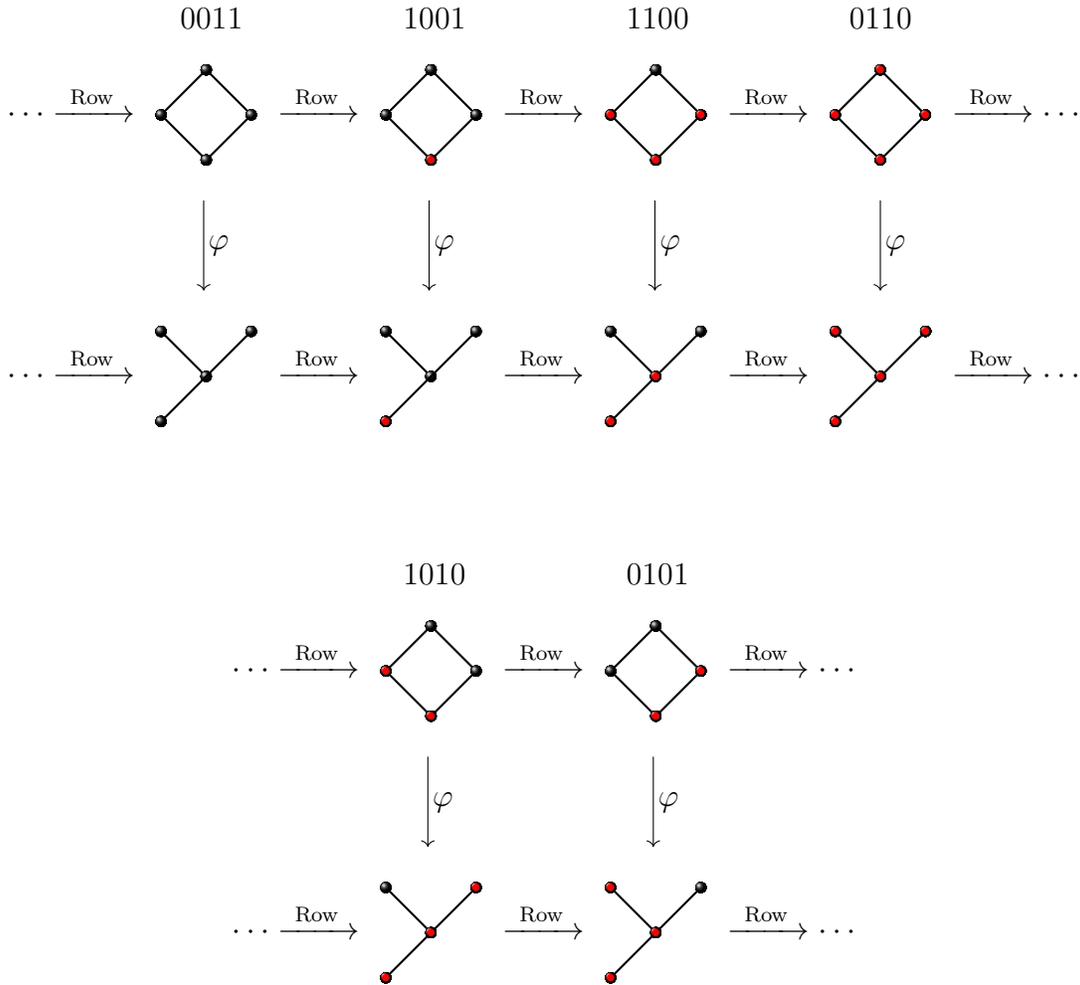
\begin{center}
$0011$\hspace{2.14cm}$1001$\hspace{2.14cm}$1100$\hspace{2.14cm}$0110$

\vspace{0.2cm}

\begin{tikzpicture}
	\node () at (0,0) {$\cdots\xrightarrow{\ \row\ }$};
	\node () at (0,-0.68){};
\end{tikzpicture}
\begin{tikzpicture}[scale=0.6]
	\SetFancyGraph
	\Vertex[NoLabel,x=0,y=0]{1}
	\Vertex[NoLabel,x=-1,y=1]{2}
	\Vertex[NoLabel,x=1,y=1]{3}
	\Vertex[NoLabel,x=0,y=2]{4}
	\Edges[style={thick}](1,2)
	\Edges[style={thick}](4,2)
	\Edges[style={thick}](3,1)
	\Edges[style={thick}](3,4)
\end{tikzpicture}
\begin{tikzpicture}
	\node () at (0,0) {$\xrightarrow{\ \row\ }$};
	\node () at (0,-0.68){};
\end{tikzpicture}
\begin{tikzpicture}[scale=0.6]
	\SetFancyGraph
	\Vertex[NoLabel,x=0,y=0]{1}
	\Vertex[NoLabel,x=-1,y=1]{2}
	\Vertex[NoLabel,x=1,y=1]{3}
	\Vertex[NoLabel,x=0,y=2]{4}
	\Edges[style={thick}](1,2)
	\Edges[style={thick}](4,2)
	\Edges[style={thick}](3,1)
	\Edges[style={thick}](3,4)
	\filldraw [red](0,0) circle (2pt);
\end{tikzpicture}
\begin{tikzpicture}
	\node () at (0,0) {$\xrightarrow{\ \row\ }$};
	\node () at (0,-0.68){};
\end{tikzpicture}
\begin{tikzpicture}[scale=0.6]
	\SetFancyGraph
	\Vertex[NoLabel,x=0,y=0]{1}
	\Vertex[NoLabel,x=-1,y=1]{2}
	\Vertex[NoLabel,x=1,y=1]{3}
	\Vertex[NoLabel,x=0,y=2]{4}
	\Edges[style={thick}](1,2)
	\Edges[style={thick}](4,2)
	\Edges[style={thick}](3,1)
	\Edges[style={thick}](3,4)
	\filldraw [red](0,0) circle (2pt);
	\filldraw [red](-1,1) circle (2pt);
	\filldraw [red](1,1) circle (2pt);
\end{tikzpicture}
\begin{tikzpicture}
	\node () at (0,0) {$\xrightarrow{\ \row\ }$};
	\node () at (0,-0.68){};
\end{tikzpicture}
\begin{tikzpicture}[scale=0.6]
	\SetFancyGraph
	\Vertex[NoLabel,x=0,y=0]{1}
	\Vertex[NoLabel,x=-1,y=1]{2}
	\Vertex[NoLabel,x=1,y=1]{3}
	\Vertex[NoLabel,x=0,y=2]{4}
	\Edges[style={thick}](1,2)
	\Edges[style={thick}](4,2)
	\Edges[style={thick}](3,1)
	\Edges[style={thick}](3,4)
	\filldraw [red](0,0) circle (2pt);
	\filldraw [red](-1,1) circle (2pt);
	\filldraw [red](1,1) circle (2pt);
	\filldraw [red](0,2) circle (2pt);
	
\end{tikzpicture}
\begin{tikzpicture}
	\node () at (0,0) {$\xrightarrow{\ \row\ }\cdots$};
	\node () at (0,-0.68){};
\end{tikzpicture}
\end{center}
\begin{center}
\ \ 
    \begin{tikzpicture}
        \draw [->] (0,0.6)--(0,-0.6);
        \node ()at (0.2,0) {$\varphi$};
    \end{tikzpicture}
    \hspace{2.16cm}
    \begin{tikzpicture}
        \draw [->] (0,0.6)--(0,-0.6);
        \node ()at (0.2,0) {$\varphi$};
    \end{tikzpicture}
    \hspace{2.16cm}
    \begin{tikzpicture}
        \draw [->] (0,0.6)--(0,-0.6);
        \node ()at (0.2,0) {$\varphi$};
    \end{tikzpicture}
    \hspace{2.16cm}
    \begin{tikzpicture}
        \draw [->] (0,0.6)--(0,-0.6);
        \node ()at (0.2,0) {$\varphi$};
    \end{tikzpicture}
\end{center}

\begin{center}
\begin{tikzpicture}
	\node () at (0,0) {$\cdots\xrightarrow{\ \row\ }$};
	\node () at (0,-0.68){};
\end{tikzpicture}
\begin{tikzpicture}[scale=0.6]
	\SetFancyGraph
	\Vertex[NoLabel,x=0,y=0]{1}
	
	\Vertex[NoLabel,x=1,y=1]{2}
	
	\Vertex[NoLabel,x=0,y=2]{3}
	\Vertex[NoLabel,x=2,y=2]{4}
	\Edges[style={thick}](1,4)
	\Edges[style={thick}](2,3)
	
\end{tikzpicture}
\begin{tikzpicture}
	\node () at (0,0) {$\xrightarrow{\ \row\ }$};
	\node () at (0,-0.68){};
\end{tikzpicture}
\begin{tikzpicture}[scale=0.6]
	\SetFancyGraph
	\Vertex[NoLabel,x=0,y=0]{1}
	
	\Vertex[NoLabel,x=1,y=1]{2}
	
	\Vertex[NoLabel,x=0,y=2]{3}
	\Vertex[NoLabel,x=2,y=2]{4}
	\Edges[style={thick}](1,4)
	\Edges[style={thick}](2,3)
	\filldraw [red](0,0) circle (2pt);
	
\end{tikzpicture}
\begin{tikzpicture}
	\node () at (0,0) {$\xrightarrow{\ \row\ }$};
	\node () at (0,-0.68){};
\end{tikzpicture}
\begin{tikzpicture}[scale=0.6]
	\SetFancyGraph
	\Vertex[NoLabel,x=0,y=0]{1}
	
	\Vertex[NoLabel,x=1,y=1]{2}
	
	\Vertex[NoLabel,x=0,y=2]{3}
	\Vertex[NoLabel,x=2,y=2]{4}
	\Edges[style={thick}](1,4)
	\Edges[style={thick}](2,3)
	\filldraw [red](0,0) circle (2pt);
	\filldraw [red](1,1) circle (2pt);
	
\end{tikzpicture}
\begin{tikzpicture}
	\node () at (0,0) {$\xrightarrow{\ \row\ }$};
	\node () at (0,-0.68){};
\end{tikzpicture}
\begin{tikzpicture}[scale=0.6]
	\SetFancyGraph
	\Vertex[NoLabel,x=0,y=0]{1}
	
	\Vertex[NoLabel,x=1,y=1]{2}
	
	\Vertex[NoLabel,x=0,y=2]{3}
	\Vertex[NoLabel,x=2,y=2]{4}
	\Edges[style={thick}](1,4)
	\Edges[style={thick}](2,3)
	\filldraw [red](0,0) circle (2pt);
	\filldraw [red](1,1) circle (2pt);
	\filldraw [red](0,2) circle (2pt);
	\filldraw [red](2,2) circle (2pt);
\end{tikzpicture}
\begin{tikzpicture}
	\node () at (0,0) {$\xrightarrow{\ \row\ }\cdots$};
	\node () at (0,-0.68){};
\end{tikzpicture}
\end{center}
\vspace{1cm}
\begin{center}
$1010$\hspace{2.14cm}$0101$

\vspace{0.2cm}

	\begin{tikzpicture}
	\node () at (0,0) {$\cdots\xrightarrow{\ \row\ }$};
	\node () at (0,-0.68){};
\end{tikzpicture}
\begin{tikzpicture}[scale=0.6]
	\SetFancyGraph
	\Vertex[NoLabel,x=0,y=0]{1}
	\Vertex[NoLabel,x=-1,y=1]{2}
	\Vertex[NoLabel,x=1,y=1]{3}
	\Vertex[NoLabel,x=0,y=2]{4}
	\Edges[style={thick}](1,2)
	\Edges[style={thick}](4,2)
	\Edges[style={thick}](3,1)
	\Edges[style={thick}](3,4)
	\filldraw [red](0,0) circle (2pt);
	\filldraw [red](-1,1) circle (2pt);
\end{tikzpicture}
\begin{tikzpicture}
	\node () at (0,0) {$\xrightarrow{\ \row\ }$};
	\node () at (0,-0.68){};
\end{tikzpicture}
\begin{tikzpicture}[scale=0.6]
	\SetFancyGraph
	\Vertex[NoLabel,x=0,y=0]{1}
	\Vertex[NoLabel,x=-1,y=1]{2}
	\Vertex[NoLabel,x=1,y=1]{3}
	\Vertex[NoLabel,x=0,y=2]{4}
	\Edges[style={thick}](1,2)
	\Edges[style={thick}](4,2)
	\Edges[style={thick}](3,1)
	\Edges[style={thick}](3,4)
	\filldraw [red](0,0) circle (2pt);
	\filldraw [red](1,1) circle (2pt);
\end{tikzpicture}
\begin{tikzpicture}
	\node () at (0,0) {$\xrightarrow{\ \row\ }\cdots$};
	\node () at (0,-0.68){};
\end{tikzpicture}
\end{center}

\begin{center}
\hspace{0.1cm}    
    \begin{tikzpicture}
        \draw [->] (0,0.6)--(0,-0.6);
        \node ()at (0.2,0) {$\varphi$};
    \end{tikzpicture}
    \hspace{2.16cm}
    \begin{tikzpicture}
        \draw [->] (0,0.6)--(0,-0.6);
        \node ()at (0.2,0) {$\varphi$};
    \end{tikzpicture}
    
\end{center}

\begin{center}
	\begin{tikzpicture}
	\node () at (0,0) {$\cdots\xrightarrow{\ \row\ }$};
	\node () at (0,-0.68){};
\end{tikzpicture}
\begin{tikzpicture}[scale=0.6]
	\SetFancyGraph
	\Vertex[NoLabel,x=0,y=0]{1}
	
	\Vertex[NoLabel,x=1,y=1]{2}
	
	\Vertex[NoLabel,x=0,y=2]{3}
	\Vertex[NoLabel,x=2,y=2]{4}
	\Edges[style={thick}](1,4)
	\Edges[style={thick}](2,3)
	\filldraw [red](0,0) circle (2pt);
	\filldraw [red](1,1) circle (2pt);
	\filldraw [red](2,2) circle (2pt);
	
\end{tikzpicture}
\begin{tikzpicture}
	\node () at (0,0) {$\xrightarrow{\ \row\ }$};
	\node () at (0,-0.68){};
\end{tikzpicture}
\begin{tikzpicture}[scale=0.6]
	\SetFancyGraph
	\Vertex[NoLabel,x=0,y=0]{1}
	
	\Vertex[NoLabel,x=1,y=1]{2}
	
	\Vertex[NoLabel,x=0,y=2]{3}
	\Vertex[NoLabel,x=2,y=2]{4}
	\Edges[style={thick}](1,4)
	\Edges[style={thick}](2,3)
	\filldraw [red](0,0) circle (2pt);
	\filldraw [red](1,1) circle (2pt);
	\filldraw [red](0,2) circle (2pt);
	
\end{tikzpicture}
\begin{tikzpicture}
	\node () at (0,0) {$\xrightarrow{\ \row\ }\cdots$};
	\node () at (0,-0.68){};
\end{tikzpicture}
\end{center}

\end{example}


The main theorem follows straightforwardly from Theorem \ref{thm:commute}; thus the goal of the rest of the paper is to prove Theorem \ref{thm:commute}. 

\begin{remark}
The bijection $\phi$ of Hamaker, Patrias, Pechenik and Williams comes from the $K$-theoretic Schubert calculus of miniscule varieties. In their paper, they defined a bijection between $\calP$-partitions of three different pairs of posets which they called \emph{minuscule \dpg\ pairs}. 
In fact, we will show $\varphi$ commutes with rowmotion on order ideals for each minuscule \dpg\ pair, with the rectangle-trapezoid pair being the hardest case.
\end{remark}

As hinted at in the preceding remark, although our main theorem concerns elementary combinatorial objects and actions, some sophisticated tools from algebra and geometry underlie our proofs, as we now explain.

In \cite{buch2008stable}, Buch et al. introduced the Hecke insertion algorithm as the $K$-theoretic analogue of the Schensted insertion algorithm. This insertion algorithm produces a class of tableaux whose entries are strictly increasing along columns and rows, hence called \textit{increasing tableaux}.
Thomas and Yong \cite{TY09} introduced a $K$-theoretic version of  Sch\"utzenberger's {\jdt} operation for increasing tableaux, which is the ``building block'' for the bijection $\varphi$.

We say a tableau is \textit{almost minimal} if its entries are at most $1$ larger than the rank of the entry.
One can realize an order ideal of a rectangle poset (resp. trapezoid poset) as an {almost minimal ordinary (resp. shifted) tableau} (\Cref{def:almost}). Then the bijection $\varphi$ applies a sequence of $K$-{\jdt} slides turning the ordinary (rectangle) tableau into a shifted (trapezoid) tableau.
The cornerstone of our proof is the \textit{(weak) $K$-Knuth equivalence} of Buch and Samuel \cite{buch2016k}, which in some sense dictates the behavior of $K$-jeu-de-taquin on shifted and ordinary tableaux. In particular, we establish the following theorem, which is the main step to proving \Cref{thm:commute}.

\begin{theorem}\label{shapeunique}
	Consider non-skew increasing tableaux.
	\begin{itemize}[nosep]
	    \item Almost minimal ordinary tableaux of the same shape are in separate $K$-Knuth equivalence classes.
	    \item Almost minimal shifted tableaux of the same shape are in separate weak $K$-Knuth equivalence classes.
	\end{itemize}
\end{theorem} 

The plan of the paper is as follows. In \Cref{background}, we review the basics of poset, tableaux and rowmotion. \Cref{jdtandhppw} surveys the $K$-{\jdt} theory and the bijection $\varphi$ of \cite{HPPW18}. \Cref{KKnuth} is devoted to the $K$-Knuth equivalence relations and a proof for \Cref{shapeunique}. In \Cref{proofmain}, we prove the main result: \Cref{thm:commute}.  Lastly, in \Cref{conjectures}, we survey some remaining conjectures of Hopkins involving rowmotion on miniscule \dpg\ pairs.

\section*{Acknowledgements}
This research was carried out as part of the 2019 REU program at the School of Mathematics of the University of Minnesota, Twin Cities. The authors are grateful for the support of NSF RTG grant DMS-1148634 for the REU program. The authors would like to thank their mentor Sam Hopkins, as well as Andy Hardt and Vic Reiner for their mentorship and guidance.  We are especially grateful to Sam Hopkins for carefully reading an earlier draft of this document and suggesting many helpful revisions.
\section{Preliminaries}\label{background}

\subsection{Partially Ordered Sets}\label{subsec:poset}
We will largely follow the convention of Stanley \cite{EC1} on partially ordered sets.

A \textit{(finite) partially ordered set} (henceforth abbreviated a \textit{poset}) is a finite set $\calP$ with a binary relation $\le$ that is reflexive, anti-symmetric, and transitive. Two elements $x,y \in \calP$ are \textit{comparable} if we have $x \le y$ or $y \le x$, and \textit{incomparable} otherwise. We say $y$ \textit{covers} $x$ or $x$ \textit{is covered by} $y$, denoted $x\lessdot y$, if $x \le y$ and there does not exist $z \in \calP$ such that $x\le z \le y$. The \textit{Hasse diagram} of poset $\calP$ is an undirected graph drawn in the plane with vertex set $\calP$, and an edge between $y$ and $x$, and $y$ drawn above $x$ if $y$ covers $x$. 

A \emph{chain} in a poset $\calP$ is a totally ordered subset of $\calP$. We say that $\calP$ is \emph{graded} if all maximal (by inclusion) chains in $\calP$ have the same size. We say that $\calP$ is \emph{ranked} if there exists a rank function $\;\rank: \calP \rightarrow \mathbb{Z}\;$ satisfying $\rank(y) = \rank(x)+1$ whenever $x \lessdot y$. We assume all rank functions are normalized so that $\min \{\rank(p): p \in \calP\}=1$, in which case a rank function is unique if it exists. Graded posets are always ranked. The posets $\mathscr R(a,b)$ and $\mathscr T(a,b)$ are examples of graded posets, where the rank of an element $(i,j)$ is $i+j-1$.

Given a poset $\calP$, an \textit{order ideal} $I$ of $\calP$ is a subset of $\calP$ that is downward closed, i.e: if $x \in I$ and $y \le x$ in $\calP$, then $y \in I$ as well.  We will often use the shorthand \textit{ideal} $I$ for an order ideal $I$. 
We denote the set of order ideals of $\calP$ by $J(\calP)$, which is itself a poset whose partial order is given by inclusion. For any subset $S$ of $\calP$, we define $\min(S)$ and $\max(S)$ to be the set of minimal and maximal elements of $S$, respectively. The set $\max(I)$ of maximal elements of an order ideal $I\in J(\calP)$ is an \textit{antichain}, i.e., a subset of $\calP$ of pairwise incomparable elements. In general, the set of order ideals of $\calP$ is in bijection with the set of antichains via the map $I \mapsto \max(I)$, with the reverse map sending an antichain $A$ to the order ideal $\langle A \rangle$ \textit{generated by} $A$, i.e.: the ideal $\langle A \rangle := \{x \in \calP \mid x \le y$ for some $y \in A\}$. 

A \emph{linear extension} of a poset $\calP$ is a total ordering $p_1\le p_2\le\cdots \le p_{\#\calP}$ of the elements of $\calP$ which extends the partial order $\le$ of $\calP$ in the sense that $p_i \le p_j$ implies $i<j$; equivalently, a linear extension is an order-preserving bijection $\rho: \calP \to \{1,2,\dots,\#\calP\}$. 
A related notion is that of $\calP$-partition: a \emph{$\calP$-partition of height m} is an order-preserving map from $\calP$ to $[m]:=\{0,1,2,...,m\}$. Denote by $\PP^m(\mathcal{P})$ the set of all $\mathcal{P}$-partitions of height m. There is a natural identification of $J(\calP)$ and $\PP^1(\mathcal{P})$ which sends an order ideal $I \in J(\calP)$ to the indicator function of its complement $\calP\setminus I$. In what follows we will often implicitly identify order ideals and height 1 $\calP$-partitions in this way.


In the next subsection, we provide preliminary definitions of increasing tableaux in order to make the connection between tableaux and order ideals.

\subsection{Young diagrams and tableaux}\label{subsec:YoungTabl}

Throughout this section let $\Lambda$ denote either the positive orthant $\mathbb{N} \times \mathbb{N}$ with order $(a_1,b_1) \leq (a_2,b_2)$ if $a_1 \leq a_2$ and $b_1 \leq b_2$, or the subset of the positive orthant $\{(a,b) \in \mathbb{N}^2| a \leq b \}$ with the induced order. We draw the positive orthant $\mathbb{N} \times \mathbb{N}$ in the French convention with
coordinates increasing from left-to-right and bottom-to-top. 
We refer to elements of
$\Lambda$ as boxes. Boxes with the same y-coordinate form a row, and
boxes with the same x-coordinate form a column.
A \textit{shape} is the set theoretic difference of boxes of the form $\lambda/\mu$ where $\mu \subsetneq \lambda$ are finite order ideals of $\Lambda$.  

An \textit{(integer) partition} $\lambda$ is a sequence $(\lambda_1,\cdots,\lambda_{\ell})$ of nonnegative integers with $\lambda_1\ge\cdots\ge\lambda_{\ell}$. 
Associated to a partition is its \textit{(ordinary) Young
diagram}, which is the shape that has $\lambda_i$ consecutive boxes in a row starting at
$(1,i)$ for $i=1,...,\ell$. (Recall that we use the French convention with boxes justified down and to
the left.) In this way partitions correspond to finite order ideals in the positive orthant.


A \textit{strict partition} $\lambda$ is a sequence $(\lambda_1,\cdots,\lambda_{\ell})$ of nonnegative integers with $\lambda_1>\cdots>\lambda_{\ell}$.
The \emph{shifted Young diagram} associated to the strict partition $\lambda$ is define similarly to its ordinary Young diagram. It has $\lambda_i$ consecutive boxes in a row starting at $(i,i)$ for $i=1,...,\ell$. In this way strict partitions correspond to finite order ideals in $\{(a,b) \in \mathbb{N}^2| a \leq b \}$.

For two partitions $\lambda=(\lambda_1,\lambda_2,\cdots)$ and $\mu=(\mu_1,\mu_2,\cdots)$, if $\mu_i \leq \lambda_i$ for all $i$, then we define the \textit{skew
(shifted) diagram} of $\lambda/\mu$ to be the set-theoretic difference of the (shifted) Young diagrams of $\lambda$ and $\mu$.
Similarly, we call a shape $\lambda/\mu$ a \textit{skew shape} when $\mu \neq \emptyset$.  
For example, the following two diagrams are, respectively, the skew ordinary diagram $\lambda\backslash\mu$ for $\lambda=(4,4,4)$ and $\mu=(3,1)$, and the skew shifted diagram for $\lambda=(4,3,2)$ and $\mu=(3,1)$. In both examples we place $\bullet$'s in the boxes belonging to $\mu$.
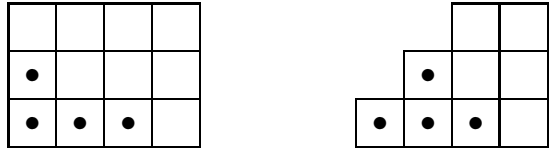
\[\begin{ytableau}
\ &\ &\ &\ \\ \bullet  &\ &\ &\ \\\bullet &\bullet &\bullet &\ 
\end{ytableau}\quad\quad\quad\quad\quad \begin{ytableau}
\none &\none &\ &\ \\ \none  &\bullet &\ &\ \\\bullet &\bullet &\bullet &\ 
\end{ytableau}\]
We will call non-skew diagrams \emph{straight}. 

A \textit{filling} of shape $\lambda/\mu \subset \Lambda$ is a function $f : \lambda/\mu \to X$ for some set $X$.
An \emph{increasing tableaux} of shape $\lambda/\mu \subset \Lambda$ is a function $T : \lambda/\mu \to S$ for a partially ordered set $S$ such that whenever $x < y$ in $\Lambda$, we have $T(x) < T(y)$. As all tableaux we consider will be increasing, we
will often drop the adjective “increasing” from now on.  When not otherwise specified, $S$ is assumed to be $\{1<2<\dots\}$.

When $S$ is the totally order set $\{1<2<\dots<m\}$ and $\Lambda$ is the positive orthant, we denote by $\IT^m(\lambda/\mu)$ the set of all increasing tableaux of shape $\lambda/\mu$ and refer to these as \textit{ordinary tableaux}. When $S$ is the totally order set $\{1<2<\dots<m\}$ and $\Lambda = \{(a,b) \in \mathbb{N}^2| a \leq b \}$, we denote by $\SIT^m(\lambda/\mu)$ the set of all increasing tableaux of shape $\lambda/\mu$ and refer to these as \textit{shifted tableaux}. Figure \ref{fig:ord_shif_tab_example} gives examples of an ordinary and a shifted tableau.
\begin{figure}[h]
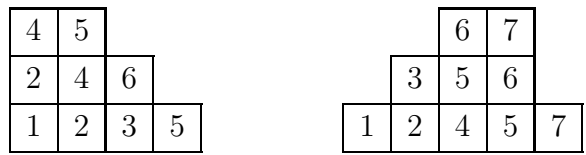

\centering
\[
\begin{ytableau}
4&5\\
2&4&6\\
1&2&3&5
\end{ytableau}\quad\quad\quad\begin{ytableau}
\none&\none&\none&6&7\\
\none&\none&3&5&6\\
\none&1&2&4&5&7
\end{ytableau}
\]
\caption{Left: an ordinary tableau of shape $(4,3,2)$, right: a shifted tableau of shape $(5,3,2)$.}
\label{fig:ord_shif_tab_example}
\end{figure}




Every shape is naturally a poset its boxes with partial order induced from $\Lambda$.  In this way, we may apply all the poset theoretic concepts from \Cref{subsec:poset} to Young diagrams.  We may now talk about \textit{poset maps} which are fillings $f:\lambda/\mu \rightarrow S$ for a partially ordered set $S$ which respect the partial order of $\lambda/\mu$ (i.e. $f(x) \leq f(y)$ whenever $x \leq y$).
$\calP$-partitions of Young diagrams are examples of poset maps.  The rank function on $\Lambda$ descends to a rank function on any Young diagram, thus we may speak of the rank of a box in a Young diagram (where we always subtract the appropriate amount so the minimal rank of a box in a Young diagram is $1$).

Observe that the rectangle poset $\mathscr R(a,b)$ is the same as the ordinary Young diagram $\lambda = (\overbrace{b,b,\dots,b}^{a\text{ times}})$ and the trapezoid poset $\mathscr T(a,b)$ is the same as the shifted Young diagram $\lambda = (a+b-1,a+b-3,\dots,b-a+1)$. \Cref{fig:rect_trap_example} gives an example of this identification for $(a,b)=(3,5)$.

\begin{figure}[h]
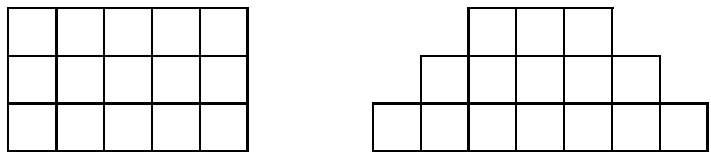

    \centering
    \ydiagram{5,5,5}
    \quad\quad\quad\quad
    \ydiagram{2+3,1+5,7}
    \caption{The posets $\mathscr R(3,5)$ and $\mathscr T(3,5)$, where the boxes are poset elements and edges in the Hasse diagram are replaced with adjacency relations.}
    \label{fig:rect_trap_example}
\end{figure}
\begin{remark}\label{rem:ordasshift}
The posets from the ordinary Young diagram $\lambda/\mu$ where $\lambda = (\lambda_n,\lambda_{n-1}, \dots, \lambda_1)$ and $\mu = (\mu_n,\mu_{n-1}, \dots, \mu_1)$ and the shifted Young diagram $\lambda'/\mu'$ where $\lambda' = (\lambda_n+(n-1),\lambda_{n-1}+(n-2), \dots, \lambda_1)$ and $\mu' = (\mu_n+n-1,\mu_{n-1}+n-2, \dots, \mu_1)$ are isomorphic. There is a bijection between $\IT^{\ell}(\lambda/\mu)$ and $\SIT^{\ell}(\lambda'/\mu')$ which realizes an ordinary tableau $T$ of shape $\lambda/\mu$ as the skew tableau $T'(x,y) = T(x+n,y)$. This observation will be important to the bijection $\varphi$.
\end{remark}

\Cref{fig:ideal_vs_tableau} depicts how an order ideal of the rectangle poset is viewed as a poset map $f:\lambda/\mu \rightarrow \{0,1\}$ of the corresponding Young diagram.

\begin{figure}[h]
\label{fig:ideal_vs_tableau}
\centering
\begin{tikzpicture}[scale=0.8,rotate=45]
	\foreach \i in {0,...,4}{
    \foreach \j in {0,1,2}
    	\filldraw (\i,\j) circle (1.6pt);}
    \foreach \t in {0,1,2}{
    \draw [-] (0,\t)--(4,\t);
    }
    \foreach \t in {0,...,4}{
    \draw [-] (\t,0)--(\t,2);
    }
    
    \foreach \i in {0,1,2,3}{
    \filldraw [red] (\i,0) circle (2.0 pt);
    }
    \foreach \i in {0,1,2}{
    \filldraw [red] (\i,1) circle (2.0 pt);
    }
    \foreach \i in {0}{
    \filldraw [red] (\i,2) circle (2.0 pt);
    }

\node at (7,-3.8) {$\begin{ytableau}
0 &1 & 1&1&1 \\
0& 0 & 0&1&1 \\
0 & 0 &0&0&1
\end{ytableau}$};
\end{tikzpicture}
\caption{On the left: an order ideal of the poset $\mathscr R(3,5)$ identified by the red color; on the right: the same order ideal in the tableau-like Hasse diagram.}
\label{fig:ideal_vs_tableau}
\end{figure}
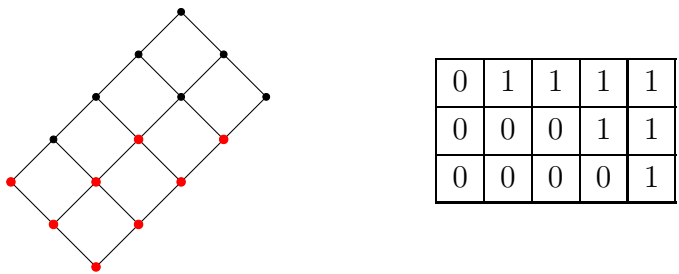

From now on, we will depict all posets as Young diagrams rather than Hasse diagrams, and will depict order ideals as $\{0,1\}$-poset maps of these Young diagrams. By abuse of notation, we denote by $\IT^{\ell}(\mathscr R(a,b))$ (resp. $\SIT^{\ell}(\mathscr T(a,b))$) the set of all ordinary (resp. shifted) tableaux with $S$ as the totally ordered set $1<2<\dots<\ell$ whose shape is the one corresponding to the poset $\mathscr R(a,b)$ (resp. $\mathscr T(a,b)$).

The \textit{minimal tableau} of a shape $\lambda/\mu$ is the tableau $T$ with $T(s) = \rank(s)$ for all boxes $s \in \lambda/\mu$. In analogy to this, we make the following definition.
\begin{definition}[Almost minimal tableaux]\label{def:almost}
	An \textit{almost minimal tableau} $T$ of shape $\lambda/\mu$ is an increasing tableau such that $T(s) - \rank(s) \in \{0,1\}$ for any boxes $s \in \lambda/\mu$. Equivalently, $T$ is obtained from an order ideal by adding rank to each entry.
	\end{definition}
\label{ideal_to_amt}
It follows from the definition that there is a bijection between the set of order ideals and the set of almost minimal tableau of that shape by adding and subtracting rank.
For example, the following tableau
\[\begin{ytableau}
3&5&6&7&8\\
2&3&4&6&7\\
1&2&3&4&6	
\end{ytableau}
\]
corresponds to the example in \Cref{fig:ideal_vs_tableau}.

\begin{remark}\label{rem:PPtoTableaux}
For an ordinary (resp. shifted) Young diagram $\lambda/\mu$ which is graded, the map which subtracts $\rank(s)$ from each entry $s$ is a bijection from $\IT^{\ell+r_{\max}}(\lambda/\mu)$ (resp. $\SIT^{\ell+r_{\max}}(\mathscr \lambda/\mu)$) to $\PP^{\ell}(\lambda/\mu)$ (This is essentially \cite[Theorem 4.7]{dilks2017resonance}; see also \cite[Section 6.2]{HPPW18}). In this situation, $\IT^{r_{\max}+1}$ (resp. $\SIT^{r_{\max}+1}$) is exactly the set of almost minimal tableaux of shape $\lambda/\mu$. Both the rectangle and trapezoid are graded posets, thus we obtain bijections \[\IT^{m+a+b-1}(\mathscr R(a,b)) \simeq \PP^m(\mathscr R(a,b))\]and \[ \SIT^{m+a+b-1}(\mathscr T(a,b)) \simeq \PP^m(\mathscr T(a,b)).\] These bijections are crucial for defining the bijection $\phi$ of \cite{HPPW18}.

\end{remark}

\subsection{Rowmotion}\label{subsec:rowmotion}

Now we define the action of rowmotion.
\begin{definition}
Let $\mathcal P$ be a poset, and $I\in J(\mathcal P)$ an order ideal of $\mathcal P$. Then the rowmotion of $I$, denoted $\row(I)$ is the order ideal generated by the minimal elements that are not in $I$, i.e.
\[\row(I)=\langle a\in\mathcal P: a\in \min\{\mathcal P\setminus I\}\rangle\]
\end{definition}
\begin{example} Here we give an example of rowmotion on an order ideal of the rectangle poset. The minimal non-elements of the initial order ideal are colored {\color{red} red}.
\[
\begin{ytableau}
0 &{\color{red}1} & 1&1&1 \\
0& 0 & 0&{\color{red}1}&1 \\
0 & 0 &0&0&{\color{red}1}
\end{ytableau}
\xrightarrow{\quad\row\quad}
\begin{ytableau}
0 &{\color{red}0} & 1&1&1 \\
0& 0 & 0&{\color{red}0}&1 \\
0 & 0 &0&0&{\color{red}0}
\end{ytableau}
\]
	
\end{example}
From this definition, it is not evident that $\row: J(\calP) \to J(\calP)$ is in fact invertible. The invertibility of $\row$, however, follows from Cameron and Fon-der-Flaass' \cite{toggles} alternative description of $\row$ using \textit{toggles}.
\begin{prop}\label{prop:row_as_toggles}
	For $p\in\mathcal P$ and order ideal $I$, we define the \emph{toggle operation} of $p$ on $I$ as follows.
\[\tau_p(I)=\begin{cases}
I\cup p&\text{if }p\notin I\text{ and } I\cup p\in J(\mathcal P)\text{,}\\
I\setminus p&\text{if }p\in I\text{ and }I\setminus p\in J(\mathcal P)\text{,}\\
I&\text{otherwise.}
\end{cases}\]
Then rowmotion is just performing toggles `row by row' \footnote{Here `row' actually refers to a rank of a poset, which is not a row but a diagonal in our Young diagram notation.} from the largest to smallest, i.e.
$$\row(I) =\tau_{p_1}\circ\tau_{p_{n-1}} \circ\cdots\circ\tau_{p_n}(I)$$
where $p_1\le\cdots\le p_n$ is any linear extension of the poset $\mathcal P$.
\end{prop}
Rowmotion is generalized by Eisenstein and Propp \cite{pwbr_row} to a piecewise linear action on $\calP$-partitions (or equivalently, order polytopes), in which toggles are given by a tropical exchange relation:
\[\tau(p)=\max\{a:a\lessdot p\}+\min\{b:p\lessdot b\}-p\]
Rowmotion on order ideals is the same as piecewise linear rowmotion on the corresponding height-$1$ $\calP$-partition. We will discuss piecewise-linear rowmotion more in \Cref{subsec:PLrow}. 

\section{$K$-{\jdt} and the Bijection $\varphi$}\label{jdtandhppw}

In this section we describe the bijection $\phi$ between $\calP$-partitions of the rectangle $\mathscr R(a,b)$ and the trapezoid $\mathscr T(a,b)$ as well as other minuscule \dpg\ pairs shown in \Cref{fig:miniscule_pairs}. The construction is based on \emph{$K$-{\jdt}} slides of Thomas and Yong \cite{TY09}.

\subsection{$K$-{\jdt} Theory for Increasing tableaux} 
\begin{definition}\label{def:swap}
   Call two boxes $s,s'$ \textit{adjacent} if $s$ covers $s'$ or $s'$ covers $s$. We define the \emph{swap} of two entries $a,b$ in a filling $f$ to be the filling $\swap_{a,b}(f)$ such that for all $x \in \calP$: \[\swap_{a,b}(f)(x) = \begin{cases}a & \text{if }f(x) = b\text{ and there exists }y\text{ adjacent to }x \text{ such that }f(y)=a,\\b & \text{if }f(x) = a\text{ and there exists }y\text{ adjacent to }x \text{ such that }f(y)=b,\\ f(x) & \text{otherwise.} \end{cases}\]
\end{definition}


Next, we can describe $K$-jeu-de-taquin as a sequence of swaps.

\begin{definition}\label{def:jdt}
    Let $T:\lambda/\mu \to \{1<2<\dots<\ell\}$ be an increasing tableaux. Let $C$ be some subset of maximal elements in $\mu$ and define $T\cup C:\lambda/\mu \cup C \to \{\bullet<1<2<\dots<\ell\}$ by 
    \[
        T\cup C(x) = \begin{cases}T(x) & x \not \in \mu \\
        \bullet & x \in C \end{cases}.
    \]
    The \emph{$K$-jeu-de-taquin forward slide} of $C$ is the tableaux 
\[\kjdt_C(T) := \left(\left(\prod_{b = 1}^\ell \swap_{\bullet,b}\right)(T \cup C)\right) \text{ with the $\bullet$s removed}.\]
   The \emph{$K$-jeu-de-taquin reverse slide} of a subset of minimal elements $C'$ in $\Lambda/\lambda$ is defined similarly by
    \[
        T\cup C'(x) := \begin{cases}T(x) & x \not \in \mu \\
        \bullet & x \in C' \end{cases}.
    \]
    and
    \[\widehat{\kjdt_{C'}}(T) := \left(\left(\prod\limits_{b = \ell}^1 \swap_{b,\bullet}\right)(T \cup C')\right)\textnormal{ with the $\bullet$s removed}.\]
    Both of the above are commonly referred to as \textit{$\kjdt$ slides}.
\end{definition}
We define the bijection $\phi$ for the case of rectangle and trapezoid. For the definition of $\varphi$ for other \dpg\ pairs, see \cite[Section 6]{HPPW18}.

\begin{definition}\label{def:hppw_bijection}
Given an increasing tableau $T \in \IT^{\ell}(\mathscr R(a,b))$ of the rectangle, one obtains an increasing tableau $\phi(T) \in \SIT^{\ell}(\mathscr T(a,b))$ as follows:
\begin{enumerate}
    \item Realize $T$ as a the skew shifted tableaux $T'$ as in \Cref{rem:ordasshift} (with $n=a$).
    \item Continually perform $\kjdt$ forward slides with $C$ as all maximal elements of the skew part until the resulting shifted tableaux is straight.
\end{enumerate}
    In other words, let $S$ be the minimal shifted tableaux of shape $(a-1,a-2,\dots,1)$ where we overline the entries of $S$ to distinguish them from the entries of $T$. Then
    \[
        \phi(T) = \kjdt_{S^{-1}(\overline{1})} \circ \kjdt_{S^{-1}(\overline{2})} \circ \hdots \circ \kjdt_{S^{-1}(\overline{2a-3})}(T').
    \]

\end{definition}

In \cite{HPPW18}, Hamaker, Patrias, pechenik and Williams proved that $\phi$ is indeed a bijection between $\IT^{\ell}(\mathscr R(a,b))$ and $\SIT^{\ell}(\mathscr T(a,b))$.  Furthermore, via the bijections $\PP^{m}(\mathscr{R}(a,b)) \simeq
\IT^{m+a+b-1}(\mathscr{R}(a,b))$ and $\PP^{m}(\mathscr T(a,b)) \simeq \SIT^{m+a+b-1}(\mathscr T(a,b))$ discussed in
\Cref{rem:PPtoTableaux}, we view $\phi$ as a bijection from $\PP^{m}(\mathscr{R}(a,b))$ to $\PP^{m}(\mathscr T(a,b))$ for all $m$.
In particular, specializing to the case $m=1$, $\phi$ is a bijection from $J(\mathscr{R}(a,b))$ to $J(\mathscr T(a,b))$.

\Cref{fig:HPPW_example} gives an example of $\phi$.
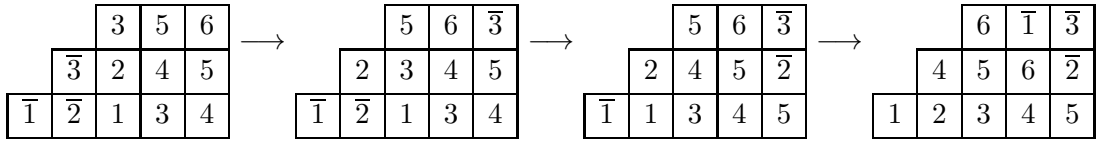
\begin{figure}[h]
    \centering
    \begin{tikzpicture}\small
    \node at (0,0) {
    \tableau{\none&\none&3&5&6\\\none&\bar{ 3}&2&4&5\\\bar{1}&\bar{2}&1&3&4}
    $\longrightarrow$
    \tableau{\none&\none&5&6&\overline{3}\\\none&2&3&4&5\\\bar{1}&\bar{2}&1&3&4}
    $\longrightarrow$
    \tableau{\none&\none&5&6&\overline{3}\\\none&2&4&5&\bar{2}\\\bar{1}&1&3&4&5}
    $\longrightarrow$
    \tableau{\none&\none&6&\bar{1}&\overline{3}\\\none&4&5&6&\bar{2}\\1&2&3&4&5}
    };
    \end{tikzpicture}
    \caption{The map $\phi$ for an order ideal $I \in J(\mathscr T(3,3))$}
    \label{fig:HPPW_example}
\end{figure}

We are interested in how $\phi$ interacts with the action of rowmotion. In the case of order ideals, we prove the following theorem.

\setcounter{section}{1}
\setcounter{restate}{0}
\begin{restate}\label{thm:main}
    For any minuscule \dpg\ pair (\Cref{fig:miniscule_pairs}) \[(P,Q) \in \{(\mathscr R(a,b),\mathscr T({a,b}),(OG(6,12),H_3),(\Q^{2n},I_2(2n))\},\] the map $\phi$ commutes with rowmotion on order ideals of the pair. In other words, we have a commutative diagram:
    \[\begin{tikzcd}
    J(P) \ar{r}{\phi} \ar{d}{\row} & J(Q) \ar{d}{\row}\\ J(P) \ar{r}{\phi} & J(Q)
    \end{tikzcd}\]
\end{restate}
\setcounter{section}{3}

\begin{figure}[htbp]
\begin{center}
\begin{tabular}{|cc|cc|} \hline 
	
	Poset Name & Hasse Diagram & Hasse Diagram  & Poset Name \\  \hline \hline 
	$\Lambda_{\Gr(k,n)}
=\mathscr R(k,n-k)$
  & \raisebox{-0.5\height}{\begin{tikzpicture}[scale=.35]
	\draw[thick] (1 cm,3) -- (2 cm,4) -- (1 cm,5) -- (0cm,4)--(1cm,3);
    \draw[thick] (3,1)--(4,2);
    \draw[thick] (3,7)--(4,6);
    \draw[thick,dotted] (1,3)--(3,1);
    \draw[thick,dotted] (2,4)--(4,2);
    \draw[thick,dotted] (1,5)--(4,2);
    \draw[thick,dotted] (4,6)--(6,4);
    \draw[thick,dotted] (4,2)--(6,4);
    \draw[thick,dotted] (5,3)--(2,6);
    \draw[thick,dotted] (2 cm,2) -- (5 cm,5);
    \draw[thick,dotted] (2 cm,4) -- (4 cm,6);
    \draw[thick,dotted] (1 cm,5) -- (3 cm,7);

    \draw[thick,solid,fill=white] (0cm,4) circle (.3cm) ;
    \draw[thick,solid,fill=white] (1cm,3) circle (.3cm) ;
    \draw[thick,solid,fill=white] (3cm,7) circle (.3cm) ;
    \draw[thick,solid,fill=white] (3cm,1) circle (.3cm) ;
    \draw[thick,solid,fill=white] (2cm,4) circle (.3cm) ;
    \draw[thick,solid,fill=white] (1cm,5) circle (.3cm) ;
    \draw[thick,solid,fill=white] (4cm,2) circle (.3cm) ;
    \draw[thick,solid,fill=white] (4cm,6) circle (.3cm) ;
    \draw[thick,solid,fill=white] (6cm,4) circle (.3cm) ;
    \draw [decorate,decoration={brace,amplitude=10pt}] (7,3.5) -- (3.5,0) node [midway, below,yshift=-0.8em,xshift=1.2em] {$n-k$};
    \draw [decorate,decoration={brace,amplitude=10pt}] (2.5,0) -- (-1,3.5) node [midway, below,yshift=-0.8em,xshift=-1.2em] {$k$};
\end{tikzpicture}} & 
\raisebox{-0.5\height}{\begin{tikzpicture}[scale=.35,rotate=270]
	\draw[thick] (4 cm,2) -- (5 cm,3)--(6,2);
    \draw[thick,dotted] (0 cm,4) -- (2 cm,6);
	\draw[thick] (0 cm,6) -- (1 cm,7);
    \draw[thick] (2 cm,6) -- (0 cm,8);
    \draw[thick,dotted] (5,3) -- (2,6);
    \draw[thick,dotted] (4,2) -- (0,6);
    \draw[thick,dotted] (3,5) -- (1,3);
    \draw[thick,dotted] (2,2) -- (0,4);
    \draw[thick,dotted] (1,5) -- (0,4);
    \draw[thick,dotted] (2,2) -- (4,4);
 
    \draw[thick,solid,fill=white] (2cm,2) circle (.3cm) ;
    \draw[thick,solid,fill=white] (0cm,6) circle (.3cm) ;
    \draw[thick,solid,fill=white] (0cm,8) circle (.3cm) ;
    \draw[thick,solid,fill=white] (0cm,4) circle (.3cm) ;
    \draw[thick,solid,fill=white] (1cm,7) circle (.3cm) ;
    \draw[thick,solid,fill=white] (4cm,2) circle (.3cm) ;
    \draw[thick,solid,fill=white] (2cm,6) circle (.3cm) ;
    \draw[thick,solid,fill=white] (5cm,3) circle (.3cm) ;
    \draw[thick,solid,fill=white] (6cm,2) circle (.3cm) ;
    \draw [decorate,decoration={brace,amplitude=10pt}] (1.5,1) -- (-1,3.5) node [midway, below,yshift=1.7em,xshift=-2em] {\scalebox{0.8}{$n-2k+1$}};
    \draw [decorate,decoration={brace,amplitude=10pt}] (0.5,9) -- (7,2.5) node [midway, above,yshift=-2em,xshift=1.2em] {$n-1$};
\end{tikzpicture}} & $\mathscr T(k,n-k)$  \\ \hline
 $\Lambda_{OG(6,12)}$ &  \raisebox{-0.5\height}{\begin{tikzpicture}[scale=.25]
    \draw[thick] (3 cm,5) -- (5 cm,7);
	\draw[thick] (3 cm,7) -- (5 cm,9);
    \draw[thick] (2 cm,8) -- (5 cm,11);
    \draw[thick] (1 cm,9) -- (4 cm,12);
    \draw[thick] (3 cm,13) -- (5 cm,11);
    \draw[thick] (3cm,11) -- (5 cm,9);
    \draw[thick] (2cm,10) -- (5cm,7);
    \draw[thick] (1cm,9) -- (4cm,6);
    \draw[thick,solid,fill=white] (3cm,5) circle (.3cm) ;
    \draw[thick,solid,fill=white] (4cm,6) circle (.3cm) ;
    \draw[thick,solid,fill=white] (5cm,7) circle (.3cm) ;
    \draw[thick,solid,fill=white] (3cm,7) circle (.3cm) ;
    \draw[thick,solid,fill=white] (4cm,8) circle (.3cm) ;
    \draw[thick,solid,fill=white] (5cm,9) circle (.3cm) ;
    \draw[thick,solid,fill=white] (2cm,8) circle (.3cm) ;
    \draw[thick,solid,fill=white] (3cm,9) circle (.3cm) ;
    \draw[thick,solid,fill=white] (4cm,10) circle (.3cm) ;
    \draw[thick,solid,fill=white] (5cm,11) circle (.3cm) ;
    \draw[thick,solid,fill=white] (1cm,9) circle (.3cm) ;
    \draw[thick,solid,fill=white] (2cm,10) circle (.3cm) ;
    \draw[thick,solid,fill=white] (3cm,11) circle (.3cm) ;
    \draw[thick,solid,fill=white] (4cm,12) circle (.3cm) ;
    \draw[thick,solid,fill=white] (3cm,13) circle (.3cm) ;
\end{tikzpicture}} & \raisebox{-0.5\height}{\begin{tikzpicture}[scale=.25]
    \draw[thick] (3 cm,9) -- (6 cm,12);
    \draw[thick] (1 cm,9) -- (5 cm,13);
    \draw[thick] (3 cm,13) -- (4 cm,14);
    \draw[thick] (1 cm,17) -- (6 cm,12);
    \draw[thick] (3 cm,13) -- (5 cm,11);
    \draw[thick] (3cm,11) -- (5 cm,9);
    \draw[thick] (2cm,10) -- (3cm,9);
    \draw[thick,solid,fill=white] (5cm,9) circle (.3cm) ;
    \draw[thick,solid,fill=white] (3cm,9) circle (.3cm) ;
    \draw[thick,solid,fill=white] (4cm,10) circle (.3cm) ;
    \draw[thick,solid,fill=white] (5cm,11) circle (.3cm) ;
    \draw[thick,solid,fill=white] (6cm,12) circle (.3cm) ;
    \draw[thick,solid,fill=white] (1cm,9) circle (.3cm) ;
    \draw[thick,solid,fill=white] (2cm,10) circle (.3cm) ;
    \draw[thick,solid,fill=white] (3cm,11) circle (.3cm) ;
    \draw[thick,solid,fill=white] (4cm,12) circle (.3cm) ;
    \draw[thick,solid,fill=white] (5cm,13) circle (.3cm) ;
    \draw[thick,solid,fill=white] (3cm,13) circle (.3cm) ;
    \draw[thick,solid,fill=white] (4cm,14) circle (.3cm) ;
    \draw[thick,solid,fill=white] (3cm,15) circle (.3cm) ;
    \draw[thick,solid,fill=white] (2cm,16) circle (.3cm) ;
    \draw[thick,solid,fill=white] (1cm,17) circle (.3cm) ;
\end{tikzpicture}} & $\Phi^+_{H_3}$  \\ \hline
$\Lambda_{\mathbb Q^{2n}}$ & \raisebox{-0.5\height}{\begin{tikzpicture}[scale=.35]
	\draw[thick,dotted] (1 cm,1) -- (3 cm,3);
    \draw[thick] (3 cm,3) -- (4 cm,4);
    \draw[thick] (2 cm,4) -- (3 cm,5);
    \draw[thick,dotted] (3 cm,5) -- (5 cm,7);
    \draw[thick] (2 cm,4) -- (3 cm,3);
    \draw[thick] (3 cm,5) -- (4 cm,4);
    \draw[thick,solid,fill=white] (1cm,1) circle (.3cm) ;
    \draw[thick,solid,fill=white] (5cm,7) circle (.3cm) ;
    \draw[thick,solid,fill=white] (3cm,3) circle (.3cm) ;
    \draw[thick,solid,fill=white] (4cm,4) circle (.3cm) ;
    \draw[thick,solid,fill=white] (2cm,4) circle (.3cm) ;
    \draw[thick,solid,fill=white] (3cm,5) circle (.3cm) ;
    \draw [decorate,decoration={brace,amplitude=10pt}] (5,3.5) -- (1.5,0) node [midway, below,yshift=-0.8em,xshift=1.2em] {$n$};
    \draw [decorate,decoration={brace,amplitude=10pt}] (1,4.5) -- (4.5,8) node [midway, above,yshift=0.7em,xshift=-1.3em] {$n$};
\end{tikzpicture}} &
\raisebox{-0.5\height}{\begin{tikzpicture}[scale=.35]
    \draw[thick] (2 cm,4) -- (4 cm,6);
    \draw[thick,dotted] (4 cm,6) -- (7 cm,9);
    \draw[thick] (3 cm,5) -- (4 cm,4);
    \draw[thick,solid,fill=white] (4cm,4) circle (.3cm) ;
    \draw[thick,solid,fill=white] (2cm,4) circle (.3cm) ;
    \draw[thick,solid,fill=white] (3cm,5) circle (.3cm) ;
    \draw[thick,solid,fill=white] (4cm,6) circle (.3cm) ;
    \draw[thick,solid,fill=white] (7cm,9) circle (.3cm) ;
    \draw [decorate,decoration={brace,amplitude=10pt}] (1,4.5) -- (6.5,10) node [midway, above,yshift=0.7em,xshift=-1.3em] {\scalebox{0.8}{$2n-1$}};
\end{tikzpicture}} & $\Phi^+_{I_2(2n)}$ \\ \hline
\end{tabular}
\end{center}
\caption{The names and Hasse diagrams of the three \dpg\ pairs considered in \cite{HPPW18}, adapted from \cite[Figure 1]{HPPW18}}
\label{fig:miniscule_pairs}
\end{figure}


Although the main theorem considers three minuscule \dpg\ pairs, the only difficulty comes from the case of the rectangle and the trapezoid. We present a proof of the other cases here, and defer the main proof to \Cref{proofmain}.

\begin{proof}[Proof for the cases of $(OG(6,12),H_3)$ and $(\Q^{2n},I_2(2n))$]

The case $(OG(6,12),H_3)$ amounts to checking finitely many applications of $\phi$ and rowmotion, which we have carried out using a computer program. For the other case, note that each of the posets $\mb Q^{2n}$ and $I_2(2n)$ only has $2$ rowmotion orbits, one of size $2n-2$ and the other of size $2$. If $I \in  J(\mb Q^{2n})$ is in the orbit of size $2$, then $\phi(I)$ is also in the orbit of size $2$ in $J(I_2(2n))$, hence rowmotion commutes with $\phi$ for such $I$. For any ideal $I$ in the other rowmotion orbit of $J(\mb Q^{2n})$, $I$ consists of all elements of $J(\mb Q^{2n})$ of rank $\leq m$ for some $m$. In this case, it can be seen that $\phi(I)$ is the set of elements of $J(I_2(n))$ of rank $\leq m$, which is the same $m$ as for $I$. Observe that for a graded poset $\calP$ with all maximal elements of $\calP$ having rank $m+1$, and the ideal $I \subseteq \calP$ consisting of all elements $\leq m$, the minimal elements of $\calP \setminus I$ are the elements of rank $m+1$, thus $\row(I)$ is either the ideal consisting of all elements of a poset $\calP$ of rank $\leq m+1$ or the empty ideal. Since $\mb Q^{2n}$ and $I_2(2n)$ are graded posets and both have the same maximal rank, we conclude that $\phi$ commutes with rowmotion on $(\mb Q^{2n},I_2(2n))$
\end{proof}

\subsection{Rowmotion via $K$-\jdt}
We can describe rowmotion as a composition of $\kjdt$ slides. This is done in \cite{dilks2017resonance} under the name $K$-promotion,
which is the $K$-theoretic analogue of Sch\"utzenberger's promotion action on standard Young tableaux.
\begin{definition}[$K$-promotion]\label{def:k-pro}
Fix a shape $\lambda/\mu$.
For an ordinary or shifted tableau $T \in \IT^{\ell}(\lambda/\mu)$ or $T \in \SIT^{\ell}(\lambda/\mu)$,
the \emph{$K$-promotion} of $T$ is a tableaux $\kpro(T) \in \IT^{\ell}(\lambda/\mu)$ or $\kpro(T) \in \SIT^{\ell}(\lambda/\mu)$ defined as follows.
\begin{enumerate}[1)]
    \item Turn the tableau into a tableau $T_1$ of shape $\lambda/(\mu \cup T^{-1}(1))$ by
    removing the minimal entry $1$ and subtracting $1$ from all other entries.
    \item Send
    \[
    T_1 \mapsto \kjdt_{T^{-1}(1)}(T_1).
    \]
    This is called \textit{$K$-rectifying} $T_1$ within $\lambda/\mu$.  Call the resulting tableau $T_2$.
    \item Add $\ell$ to boxes such that the resulting tableaux is the shape $\lambda/\mu$. In other words, $K$-pro$(T)$ is the tableau with shape $\lambda/\mu$ with
    \[
        \kpro(T) = \begin{cases} 
        T_2(s) & s \textnormal{ is in the domain of } T_2 \\
        \ell & \textnormal{otherwise}
        \end{cases}.
    \]
\end{enumerate}
In the case where there does not exist a $1$ in the tableau, $K$-promotion simply decrements each entry by $1$.
\end{definition}

Recall from \Cref{rem:PPtoTableaux} that for (ordinary or shifted) shapes that are
graded, almost minimal
tableaux are exactly the tableaux in $\IT^{r_{max}+1}(\lambda/\mu)$ or $\SIT^{r_{max}+1}(\lambda/\mu)$. Rowmotion is related to
$K$-jeu-de-taquin via the following lemma:
\begin{lemma}
\label{lem:k-pro}
For an ordinary (resp. shifted) Young diagram $\lambda/\mu$ which is graded as a poset, $K$-promotion on $\IT^{r_{max}+1}(\lambda/\mu)$ (resp. $\SIT^{r_{max}+1}(\lambda/\mu)$) is equivariant to the inverse action of rowmotion on the corresponding order ideals.
\end{lemma}


Dilks, Pechenik and Striker prove a slightly different statement; specifically on the rectangle, a flip of our bijection between $\PP^{\ell}(\mathscr R(a,b))$ and $\IT^{\ell+a+b-1}(\mathscr R(a,b))$ described in Remark~\ref{rem:PPtoTableaux} is an intertwining operator between rowmotion on $J(\mathscr R(a,b) \times [\ell])$ and $K$-promotion on $\IT^{\ell+a+b-1}(\mathscr R(a,b))$ \cite[Lemma 4.2]{dilks2017resonance}. In \cite{dilks2019rowmotion}, Dilks, Striker and Vorland generalize this result.  Dilks, Pechenik and Striker's intertwining operator implies Lemma~\ref{lem:k-pro} in the rectangle case and the tools they introduce we will use to prove this statement for all graded $\lambda/\mu$. Namely they introduce the \textit{$K$-Bender-Knuth involutions}
\[\kbk_i(T)(x) := \begin{cases}i & \text{if }f(x) = i+1\text{ and there does not exist }y < x \text{ such that }f(y)=i,\\i+1 & \text{if }f(x) = i\text{ and there does not exist }y>x \text{ such that }f(y)=i+1,\\ f(x) & \text{otherwise.} \end{cases}\]
and show that for $T \in \IT^{\ell}(\lambda/\mu)$ or $T \in \SIT^{\ell}(\lambda/\mu)$
\[
\kpro(T) = \kbk_{\ell-1} \circ \kbk_{\ell-2} \circ \dots \circ \kbk_{1} (T)
\]
\cite[Proposition 2.5]{dilks2017resonance}. For a graded poset, define $H_i$ to be product of all toggles $\tau_p$ where $p$ has rank $i$. From the alternate description of rowmotion in terms of toggles (see Proposition~\ref{prop:row_as_toggles}), inverse rowmotion can be written
\[
    \row^{-1} = H_{r_{max}}\circ H_{r_{max} - 1} \circ \dots \circ H_1.
\]
\begin{proof}
    By \cite[Proposition 2.5]{dilks2017resonance} and the toggle description of rowmotion, 
    it is enough to show for ordinary or shifted shapes, the following diagram commutes for $1\leq i \leq r_{max}$
    \[\begin{tikzcd}
    J(\lambda/\mu) \ar{r}{\Psi} \ar{d}{H_i} & \IT^{r_{max}+1}(\lambda/\mu) \text{ or } \SIT^{r_{max}+1}(\lambda/\mu)\ar{d}{\kbk_{i}}\\ J(\lambda/\mu) \ar{r}{\Psi} & \IT^{r_{max}+1}(\lambda/\mu) \text{ or } \SIT^{r_{max}+1}(\lambda/\mu)
    \end{tikzcd}\]
    where $\Psi$ is the bijection from Remark~\ref{rem:PPtoTableaux}.  Equivalently, viewing order ideals as as $\{0,1\}$-poset maps and $\Psi$ as adding rank, for any $\{0,1\}$-poset map $f$ on $\lambda/\mu$ and any square $s \in \lambda/\mu$,
    \[
    \Psi \circ H_i (f)(s) = \kbk_{i} \circ \Psi (f)(s).
    \]
    Case 1: $\Psi (f)(s) \not \in \{i,i+1\}$: 
    \begin{itemize}[label={}]
        \item Since $\rank(s) \neq i$, $H_i (f)(s) = f(s)$.
        \item Since $\Psi (f)(s) \not \in \{i,i+1\}$, $\kbk_{i} \circ \Psi (f)(s) = \Psi (f)(s)$.
    \end{itemize}
    Case 2: $\rank(s) = i-1$ and $f(s) = 1$: 
    \begin{itemize}[label={}]
    \item Since $\rank(s) \neq i$, $H_i (f)(s) = f(s)$.
    \item For any square $s'$ covering $s$, $f(s') = 1$ and $\Psi (T)(s') = i+1$, thus $\kbk_{i} \circ \Psi (f)(s) = \Psi (f)(s)$.
    \end{itemize}
    Case 3: $\rank(s) = i$ and $f(s) = 1$: 
    \begin{itemize}[label={}]
    \item If for some square $s'$ covered by $s$, $f(s') = 1$, then  $H_i (f)(s) = 1$ and $\kbk_{i} \circ \Psi (f)(s) = i+1$. 
    \item Otherwise, $H_i (f)(s) = 0$ and $\kbk_{i} \circ \Psi (f)(s) = i$.
    \end{itemize}
    Case 4: $\rank(s) = i$ and $f(s) = 0$:
    \begin{itemize}[label={}]
    \item If for some square $s'$ covering $s$, $f(s') = 0$, then  $H_i (f)(s) = 0$ and $\kbk_{i} \circ \Psi (f)(s) = i$. 
    \item Otherwise, $H_i (f)(s) = 1$ and $\kbk_{i} \circ \Psi (f)(s) = i+1$.
    \end{itemize}
    Case 5: $\rank(s) = i+1$ and $f(s) = 0$:
    \begin{itemize}[label={}]
    \item Since $\rank(s) \neq i$, $H_i (f)(s) = f(s)$.
    \item For any square $s'$ covered by $s$, $f(s') = 0$ and $\Psi (T)(s') = i$, thus $\kbk_{i} \circ \Psi (f)(s) = \Psi (f)(s)$.
    \end{itemize}
\end{proof}
\begin{remark}
    As in Lemma 4.2 in \cite{dilks2017resonance}, the above argument can be extended to show $\Psi$ in an intertwining operator between inverse rowmotion on $J(\lambda/\mu \times [\ell])$ and $K$-promotion on $\IT^{\ell+a+b-1}(\lambda/\mu)$ (resp. $\SIT^{\ell+a+b-1}(\lambda/\mu)$) for an ordinary (resp. shifted) tableau $\lambda/\mu$ which is graded.
\end{remark}

\begin{example}[Inverse rowmotion as $K$-promotion]\label{ex:K-pro}

\begin{align*}
   \begin{ytableau}
4 & 5 & 6 \\
2 & 4 & 5 \\
1 & 2 & 3
\end{ytableau}\xrightarrow{\text{  remove }1}
\begin{ytableau}
3 & 4 & 5 \\
1& 3 & 4 \\
\bullet & 1 & 2
\end{ytableau}\xrightarrow{\ \text{ }K\text{-rectify  }}
\begin{ytableau}
4 & 5 & \bullet \\
3& 4 & 5 \\
1 & 2 & 4
\end{ytableau}
\xrightarrow{\text{add maximum}}
\begin{ytableau}
4 & 5 & 6 \\
3& 4 & 5 \\
1 & 2 & 4
\end{ytableau}\\ \\
\begin{ytableau}
4 & 5 & 6 \\
2& 4 & 5 \\
1 & 2 & 3
\end{ytableau}
\xleftarrow{\text{ add rank  }}
\begin{ytableau}
1 & 1 & 1 \\
0& 1 & 1 \\
0 & 0 & 0
\end{ytableau}
 \xleftarrow{\text{Rowmotion}} \begin{ytableau}
1 & 1 & 1 \\
1& 1 & 1 \\
0 & 0 & 1
\end{ytableau} \xleftarrow{\text{subtract rank }} 
\begin{ytableau}
4 & 5 & 6 \\
3& 4 & 5 \\
1 & 2 & 4
\end{ytableau}
\end{align*}

\end{example}

\section{$K$-Knuth and weak $K$-Knuth Equivalence}\label{KKnuth}
The connection between the bijection of $\phi$ and rowmotion is most apparent in their use of $K$-{\jdt} slides.  In this section, we will introduce some invariants of $K$-{\jdt} and use these invariants to prove Theorem~\ref{thm:commute} through considering rowmotion and $\phi$ in terms of $K$-{\jdt} slides.
\subsection{$\kjdt$ equivalence for ordinary and shifted tableaux}\label{subsec:K-jdt=Knuth}
Using the forward and reverse $\kjdt$ slides in Definition~\ref{def:jdt}, we may define an equivalence relation on tableaux in $\Lambda$ for a fixed $\Lambda$.
For a fixed poset $\Lambda$, two tableaux $T,T'$ in $\Lambda$ are considered \emph{$\kjdt$ equivalent} if $T$ can be reached from $T'$ by a series of $\kjdt$ slides.  We will be concerned with $\kjdt$ equivalence of ordinary and shifted tableaux. 
In particular, in the bijection $\varphi$, the realization $T'$ of a tableaux of the rectangle as a skew shifted tableaux is $\kjdt$ equivalent to $\varphi(T')$ since $\varphi$ can be written as a the composition of $\kjdt$ slides. 

\begin{prop}\label{prop:phiequiv}
    $T'$ and $\varphi(T)$ are $\kjdt$ equivalent shifted tableaux.
\end{prop}

For an increasing tableaux $T: \lambda/\mu \rightarrow \{1 < 2 < \dots < m\}$, define the tableaux $T|_{[a,b]}$ as the restriction of $T$ to $T^{-1}([a,b]) \subseteq \lambda/\mu$.
A useful consequence of the way that we perform the swaps, is that if we restrict two $\kjdt$ equivalent tableaux $T,T'$ to the same interval $[a,b]$, then a slight modification (essentially a restriction) of the $\kjdt$ slides which change $T$ into $T'$ will change $T|_{[a,b]}$ into $T'|_{[a,b]}$. Specifically,
\begin{lemma}\cite[Lemma 3.3]{buch2016k}
\label{lem:ResInterval}
If $T$ and $T'$ are $\kjdt$ equivalent, then $T|_{[a,b]}$ and $T'|_{[a,b]}$ are $\kjdt$ equivalent.
\end{lemma}

In \cite{buch2016k}, Buch and Samuel show that $\kjdt$ equivalence for ordinary and shifted tableaux can be described by $K$-Knuth and weak $K$-Knuth equivalence relations of their reading words, respectively.

\begin{definition}
    The \emph{row reading word} of a tableau $T$ of ordinary or shifted shape 
    is the word obtained by reading the rows of $T$ from top to bottom, and from right to left within each row.
\end{definition}
\begin{example}
\begin{align*}
   \begin{ytableau}
4 & 5 & 6 \\
2 & 4 & 5 \\
1 & 2 & 3
\end{ytableau} \quad\quad& \textnormal{ has row reading word } 456245123 \\
\begin{ytableau}
\none & \none & 6 & \none & \none \\
\none & 3 & 4 & 5 & \none \\
1 & 2 & 3 & 4 & 6
\end{ytableau}\hspace{0.17cm} & \textnormal{ has row reading word } 634512346 \\
\end{align*}
\end{example}

\begin{theorem}\cite[Theorem 6.2]{buch2016k}
    Ordinary tableaux $T,T'$ are $\kjdt$ equivalent if and only if their row reading words are $K$-Knuth equivalent, where $K$-Knuth equivalence is the symmetric transitive closure of the following basic equivalences:
    \begin{enumerate}[nosep]
        \item[$\circ$] $uaav \equiv uav$ for integers $a$ and words $u,v$,
        \item[$\circ$] $uabav \equiv ubabv$ for integers $a,b$ and words $u,v$,
        \item[$\circ$] $uabcv \equiv uacbv$ for integers $b< a <c$ and words $u,v$,
        \item[$\circ$] $uabcv \equiv ubacv$ for integers $a < c <b$ and words $u,v$.
    \end{enumerate}
\end{theorem}

\begin{theorem}\cite[Theorem 7.8]{buch2016k}
Shifted tableaux $T,T'$ are $\kjdt$ equivalent if and only if their row reading words are weakly $K$-Knuth equivalent, where weak $K$-Knuth equivalence is the symmetric transitive closure of the basic equivalences of $K$-Knuth equivalence and the following basic equivalence 
\begin{itemize}[nosep]
    \item[$\circ$] $abv \equiv bav$ for integers $a,b$ and word $v$.
\end{itemize}
\end{theorem}

In light of the above theorems, we say two ordinary tableaux are \textit{$K$-Knuth equivalent} if their row reading words are $K$-Knuth equivalent and we say two shifted tableaux are \textit{weak $K$-Knuth equivalent} if their row reading words are weak $K$-Knuth equivalent.

The reader may notice that $K$-Knuth and weak $K$-Knuth equivalence are similar and hence believe that $\kjdt$ equivalences of ordinary and shifted tableaux are related. This is true. For a shifted tableau $T$, we may construct an ordinary tableau by reflecting $T$ across the diagonal. Concretely, we define $T^2$ to be the ordinary tableau with boxes $(i,j),(j,i)$ for boxes $(i,j)$ of $T$, where 
\[
T^2(i,j) := \begin{cases}
T(i,j) & i \leq j \\
T(j,i) & i>j
\end{cases}.
\]

\begin{prop}\cite[Proposition 7.1]{buch2016k}
\label{prop:7.1BuchSamuel}
    If $T$ and $T'$ are $\kjdt$ equivalent shifted tableaux, then $T^2$ and $T'^2$ are $\kjdt$ equivalent ordinary tableaux.
\end{prop}

Buch and Samuel also conjecture that the converse of the above proposition is true.

\subsection{Hecke permutations}
While weak $K$-Knuth and $K$-Knuth equivalence of row reading words completely describe $\kjdt$ equivalence, these equivalences can be difficult to work with. Buch and Samuel \cite{buch2016k} introduce a simpler yet cruder invariant of $\kjdt$ on ordinary tableaux and use this invariant to prove minimal tableaux are \textit{unique rectification targets} (unique in their (weak) $K$-Knuth class among all straight tableaux). This invariant is called the Hecke permutation. The \emph{Hecke product} of a permutation $u$ and a simple transposition $s_i = (i,i+1)$ is denoted $u \cdot s$ with
\[
    u \cdot s_i = \begin{cases}
    u    & \textnormal{ if } u(i)>u(i+1)\\
    us_i    & \textnormal{ if } u(i)<u(i+1)
    \end{cases}
\]

\begin{definition}[\cite{buch2016k}]
    The \textit{Hecke permutation} of a tableau $T$ with reading word $u = a_1a_2a_3 \hdots a_k$, is the Hecke product
    \[
        s_{a_k} \cdot (s_{a_{k-1}} \cdot (s_{a_{k-2}} \hdots (s_{a_2} \cdot s_{a_1}) \hdots ))
    \]
    which is a permutation on $\max(a_1,a_2 \hdots, a_k)+1$ elements. We will denote this permutation by $w(T)$ or $w(u)$.
\end{definition}

If two tableaux reading words $u$ and $u’$ are $K$-Knuth equivalent, then we have $w(u) = w(u’)$ (although the converse need not be true). In particular, this implies

\begin{corollary}\cite[Corollary 6.5]{buch2016k}
\label{cor:Heckeinvariant}
    The Hecke permutation of an ordinary tableau is invariant under $\kjdt$ slides. 
\end{corollary}

\subsection{Minimal ideals and $K$-Knuth equivalence}

Although our proof of Theorem~\ref{thm:commute} will only involve almost minimal tableaux, our main theorem in this subsection will be more general and we will need more general notations:

\begin{definition}
    Given an increasing tableau $T$ of shape $\lambda/\mu$, its \emph{minimal ideal} is the set of boxes $s$ such that $T(s)-\rank(s)=0$. This set is downward closed and thus is an order ideal of the poset $\lambda/\mu$. For convenience will denote the \emph{minimal ideal} of $T$ by $I_0$ and the minimal ideal of $T'$ by $I_0'$.
\end{definition}

Our results in this section will come from analyzing the Hecke permutations of ordinary tableaux. We will specifically be interested in finding where elements $i$ occur in the Hecke permutation formed by the row reading word of a tableau.

\begin{prop}
\label{prop:HeckeDownr}
Let $T$ be a tableau and $\overline{T_r}$ be the tableau $T$ without the first $r$ rows, then for any $i$, $w(T)^{-1}(i) \geq w(\overline{T_r})^{-1}(i)-r$.
\end{prop}
\begin{proof}
Let $m = w(\overline{T_r})^{-1}(i)$. Each time we compute the Hecke product of $w$ with a transposition $s_n$, only the $n$-th and $(n+1)$-th entries of $w$ is changed. Since each row of $T$ is increasing, the entry $m-1$ only appears at most once in the $r$-th row. Thus 
\[
w(\overline{T_{r-1}})^{-1}(i) \geq w(\overline{T_r})^{-1}(i) - 1.
\]
The proposition now follows by induction.
\end{proof}

In the following lemma, $R_i(T)$ denotes the $i$th row of the tableau $T$.  
\begin{lemma}
\label{lem:HeckeMinIdeal}
    Let $I_0$ be the minimal ideal of a straight ordinary tableau $T$. 
    \begin{enumerate}[(i)]
        \item If $|R_i(T) \cap I_0| < |R_{i-1}(T) \cap I_0|$, then $w(T)^{-1}(i)=|R_i(T) \cap I_0|+1$
        \item If $|R_i(T) \cap I_0| = |R_{i-1} \cap I_0|$, then $w(T)^{-1}(i)>|R_i(T) \cap I_0|+1$
    \end{enumerate}
\end{lemma}
\begin{proof}
For each $r$ with $|R_r(T) \cap I_0| > 0$, the first element in the $r$-th row is the first appearance of $r$ in the row reading word of $T$. Using this fact, we can see that the $|R_r(T) \cap I_0|$ part of the reading word in $w(\overline{T_{r-1}})$ transposes the element $r$ with its neighbor successively $|R_r(T) \cap I_0|$ times. When $|R_r(T) \cap I_0| = 0$, there is no $r$ in the tableaux  $\overline{T_{r-1}}$. This yields
\begin{equation}
\label{eq:HeckeFirstOccurence}
w(\overline{T_{r-1}})^{-1}(r) = |R_r(T) \cap I_0|+r.
\end{equation}
For any integer $a$, if 
\begin{align}
\label{eq:HeckeMoveDownOne1}
r<w(\overline{T_{r}})^{-1}(a) \leq |R_r(T) \cap I_0|+r,
\end{align}
then
\begin{align}
\label{eq:HeckeMoveDownOne2}
w(\overline{T_{r-1}})^{-1}(a) = w(\overline{T_{r}})^{-1}(a)-1.
\end{align}
(i) We will prove by induction that for all $j\leq i$
\[
w(\overline{T_{j-1}})^{-1}(i) = |R_i(T) \cap I_0|+j.
\] This will be enough to prove (i). Our induction here is on $j$, with the base case as $j=i$ which follows from equation \ref{eq:HeckeFirstOccurence}. For the inductive step, suppose for $j\leq i$
\[
w(\overline{T_{j-1}})^{-1}(i) = |R_i(T) \cap I_0|+j.
\]
By our assumption, $|R_i(T)\cap I_0| < |R_{i-1}(T)\cap I_0|$ and since $I_0$ is an ideal,  $|R_{i-1}(T)\cap I_0|< |R_{j-1}(T)\cap I_0|$. Thus we have
\[
j-1<|R_i(T) \cap I_0|+j \leq |R_{i-1}(T)\cap I_0|+j-1 \leq |R_{j-1}(T) \cap I_0|+j-1.
\]
Finally our argument around equations \ref{eq:HeckeMoveDownOne1} and \ref{eq:HeckeMoveDownOne2} finish our inductive step.\\
(ii) Equation \ref{eq:HeckeFirstOccurence} implies,
\[
w(\overline{T_{i-1}})^{-1}(i) = |R_i(T) \cap I_0|+i.
\]
and using our assumption $|R_i(T) \cap I_0| = |R_{i-1} \cap I_0|$, we see that $|R_{i-1} \cap I_0| + i - 1 = w(\overline{T_{i-1}})^{-1}(i)-1$ is not in the $i-1$-th row of $T$. Thus 
\[
    w(\overline{T_{i-2}})^{-1}(i)\geq w(\overline{T_{i-1}})^{-1}(i),
\]
and by Proposition \ref{prop:HeckeDownr}, we conclude
\[
w(T)^{-1}(i)\geq |R_i(T) \cap I_0|+2.
\]
\end{proof}

\begin{theorem}[ordinary shape]
\label{thm:kjdtequivIdeal}
Let $T$ and $T'$ be $\kjdt$ equivalent straight ordinary tableaux with minimal ideals $I_0$ and $I_0'$. Then $I_0=I_0'$.
\end{theorem}
\begin{proof}
Suppose that $I_0 \neq I_0'$. Let $r$ be the first row where $I_0$ and $I_0'$ differ, then \Cref{lem:HeckeMinIdeal} implies that 
\[
    w(T)^{-1}(r) \neq w(T')^{-1}(r).
\]
Therefore the Hecke permutations of $T$ and $T'$ differ. Since by Corollary \ref{cor:Heckeinvariant}, Hecke permutations are invariant under $\kjdt$ slides for ordinary tableaux, $T$ and $T'$ are not $\kjdt$ equivalent.
\end{proof}

Since almost minimal tableaux are completely described by their shape and their minimal ideal, as a corollary we conclude: 
\setcounter{section}{1}
\setcounter{restate}{2}
\begin{restate}[ordinary shape]
\label{cor:uniqueidealsTypeA}
    For any (weak) partition $\lambda$, all almost-minimal tableaux of shape $\lambda$ are in separate $K$-Knuth equivalence classes.
\end{restate}

To extend the above two results to shifted tableaux, we will use the connection bewteen $\kjdt$ of the shifted tableau $T$ and $\kjdt$ of the ordinary tableau $T^2$. Notice in an ordinary tableau, $\rank(i,j) = \rank(j,i)$. It follows for a straight shifted tableau $T$, for any box $s=(i,j) \in T^2$, $T^2(s)-\rank(s) = T^2[(j,i)]-\rank(j,i)$. Thus $T$ is almost minimal if and only if $T^2$ is almost minimal and two minimal ideals of tableaux $T$ and $T'$ are equal if and only if the minimal ideals of $T^2$ and $T'^2$ are equal.  Our above ordinary shape results combined with these observations and Proposition \ref{prop:7.1BuchSamuel} imply that:
\setcounter{section}{4}
\setcounter{restate}{11}
\begin{restate}[shifted shape]
\label{thm:kjdtequivIdealShifted}
    Let $T$ and $T'$ be $\kjdt$ equivalence straight shifted tableaux with minimal ideals $I_0$ and $I_0'$. Then $I_0 = I_0'$.
\end{restate}
\setcounter{section}{1}
\setcounter{restate}{2}
\begin{restate}[shifted shape]
\label{cor:uniquetrapideals}
    For any strict partition $\lambda$, all almost-minimal tableaux of shape $\lambda$ are in separate weak $K$-Knuth equivalence classes.
\end{restate}
\setcounter{section}{4}

\begin{remark}
    A \textit{unique rectification target} is a straight tableau $T$ such that it is the only straight tableau in its $\kjdt$ equivalence class.  Unique rectification targets are crucial to the $K$-theoretic origins of $\kjdt$ in \cite{TY09}.  They have been further studied in \cite{buch2016k} and \cite{gaetz2016k}, the former of whom showed that minimal tableaux are unique rectification targets.  Similar to \cite[Proposition 2.43]{gaetz2016k}, Theorem~\ref{thm:kjdtequivIdeal} describes an invariant of $\kjdt$ rectifications and could have nice implications for unique rectification targets.  One might ask if almost minimal tableaux are unique rectification targets but this is not always the case, as the below example shows:
    \begin{example}\cite[Example 7.4]{gaetz2016k}
All the tableaux below are in the same $K$-Knuth equivalence class.
\begin{align*}
   \begin{ytableau}
5 \\
4 \\
2 \\
1 & 2 & 3 & 5 & 6
\end{ytableau} \hspace{.3 cm}
   \begin{ytableau}
5 \\
4 \\
2 & 4 \\
1 & 2 & 3 & 5 & 6
\end{ytableau} \hspace{.3 cm}
   \begin{ytableau}
5 \\
4 \\
2 & 5 \\
1 & 2 & 3 & 5 & 6
\end{ytableau} \hspace{.3 cm}
   \begin{ytableau}
5 \\
4  \\
2 & 4 & 5\\
1 & 2 & 3 & 5 & 6
\end{ytableau}
\end{align*}
\end{example}
\end{remark}

The last ingredient we will need to prove that the bijection of \cite{HPPW18} commutes with rowmotion on order ideals is the following corollary:
\begin{corollary}
\label{cor:intervalImpliesEqual}
Let $T,T'$ be two almost minimal (ordinary or shifted) tableaux of the same shape with maximal rank $r$. Then $T|_{[1,r]}$ is $\kjdt$ equivalent to $T'|_{[1,r]}$ if and only if $T=T'$.
\end{corollary}
\begin{proof}
($\Leftarrow$) If $T=T'$, then $T|_{[1,r]} = T'|_{[1,r]}$.\\
($\Rightarrow$) Suppose $T|_{[1,r]}$ is $\kjdt$ equivalent to $T'|_{[1,r]}$. Let $I_0$ be the minimal ideal of $T$ and $I_0'$ the minimal ideal of $T'$. Notice that $I_0,I_0'$ are also the minimal ideals of $T|_{[1,r]}$ and $T'|_{[1,r]}$, respectively. Then by Theorem~\ref{thm:kjdtequivIdeal}, $I_0 = I_0'$. Since almost minimal tableaux are completely determined by their shape and their ideal, $T=T'$.
\end{proof}

Recall from Proposition~\ref{prop:phiequiv} that
the bijection $\phi: \IT^{\ell}(\mathscr R(a,b)) \to \SIT^{\ell}[\mathscr T(a,b)]$ of \cite{HPPW18} preserves $\kjdt$
equivalence (and hence weak K-Knuth equivalence). Recall, as described in \Cref{subsec:YoungTabl}, that
we may restrict $\phi$ to almost minimal tableaux, a.k.a., order ideals, to get a bijection
$\phi\colon J(\mathscr R(a,b)) \to J(\mathscr T(a,b))$. Then we have the following corollary of \Cref{cor:uniquetrapideals}:
\begin{corollary}
\label{lem:hppwuniquematch}
The bijection $\phi$ acting on order ideals can be described as matching each almost minimal tableau of the rectangle with its unique weak $K$-Knuth equivalent almost minimal tableau of the trapezoid.
\end{corollary}

\section{Commuting of Rowmotion and $\varphi$}\label{proofmain}
We have now built up the machinery to finally prove \Cref{thm:commute}.

Using the definition of rowmotion as inverse $K$-promotion, we will show that performing the above process preserves weak $K$-Knuth equivalence of order ideals of the rectangle and trapezoid i.e. if $I$ is weakly $K$-Knuth equivalent to $J$, then $\row^{-1}(I)$ is weakly $K$-Knuth equivalent to $\row^{-1}(J)$. By Corollary \ref{lem:hppwuniquematch} this implies that rowmotion inverse commutes with $\phi$ and thus rowmotion commutes with $\phi$.
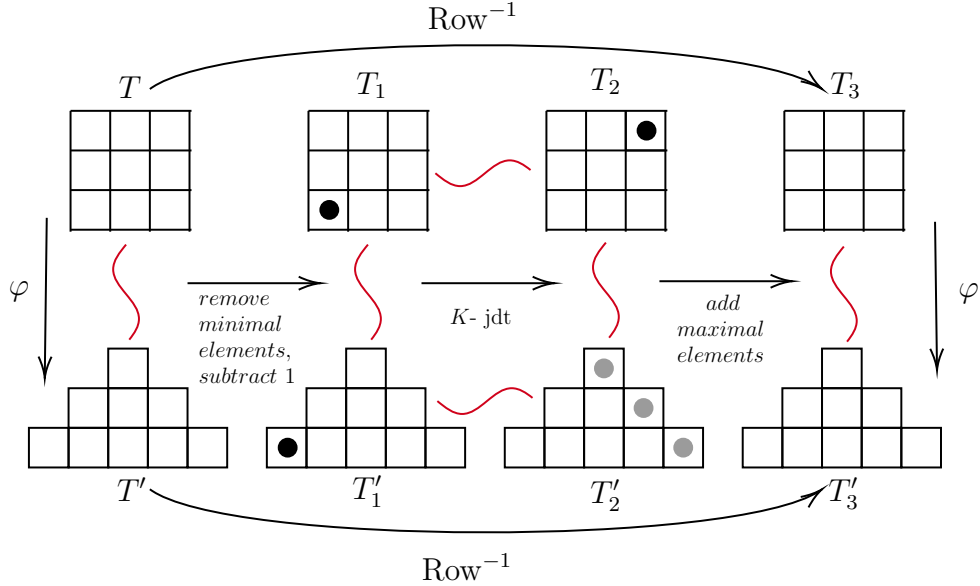
\begin{figure}[h]
    \centering
\tikzset{every picture/.style={line width=0.75pt}}
\begin{tikzpicture}[x=0.75pt,y=0.75pt,yscale=-1,xscale=1]
\draw  [draw opacity=0] (120,127) -- (181,127) -- (181,188) -- (120,188) -- cycle ; \draw   (120,127) -- (120,188)(140,127) -- (140,188)(160,127) -- (160,188)(180,127) -- (180,188) ; \draw   (120,127) -- (181,127)(120,147) -- (181,147)(120,167) -- (181,167)(120,187) -- (181,187) ; \draw    ;
\draw  [draw opacity=0] (360,127) -- (421,127) -- (421,188) -- (360,188) -- cycle ; \draw   (360,127) -- (360,188)(380,127) -- (380,188)(400,127) -- (400,188)(420,127) -- (420,188) ; \draw   (360,127) -- (421,127)(360,147) -- (421,147)(360,167) -- (421,167)(360,187) -- (421,187) ; \draw    ;
\draw  [draw opacity=0] (240,127) -- (301,127) -- (301,188) -- (240,188) -- cycle ; \draw   (240,127) -- (240,188)(260,127) -- (260,188)(280,127) -- (280,188)(300,127) -- (300,188) ; \draw   (240,127) -- (301,127)(240,147) -- (301,147)(240,167) -- (301,167)(240,187) -- (301,187) ; \draw    ;
\draw  [draw opacity=0] (480,127) -- (541,127) -- (541,188) -- (480,188) -- cycle ; \draw   (480,127) -- (480,188)(500,127) -- (500,188)(520,127) -- (520,188)(540,127) -- (540,188) ; \draw   (480,127) -- (541,127)(480,147) -- (541,147)(480,167) -- (541,167)(480,187) -- (541,187) ; \draw    ;
\draw   (99,287) -- (119,287) -- (119,307) -- (99,307) -- cycle ;
\draw   (119,287) -- (139,287) -- (139,307) -- (119,307) -- cycle ;
\draw   (139,287) -- (159,287) -- (159,307) -- (139,307) -- cycle ;
\draw   (159,287) -- (179,287) -- (179,307) -- (159,307) -- cycle ;
\draw   (179,287) -- (199,287) -- (199,307) -- (179,307) -- cycle ;
\draw   (119,267) -- (139,267) -- (139,287) -- (119,287) -- cycle ;
\draw   (139,267) -- (159,267) -- (159,287) -- (139,287) -- cycle ;
\draw   (159,267) -- (179,267) -- (179,287) -- (159,287) -- cycle ;
\draw   (139,247) -- (159,247) -- (159,267) -- (139,267) -- cycle ;
\draw   (239,287) -- (259,287) -- (259,307) -- (239,307) -- cycle ;
\draw   (259,287) -- (279,287) -- (279,307) -- (259,307) -- cycle ;
\draw   (279,287) -- (299,287) -- (299,307) -- (279,307) -- cycle ;
\draw   (299,287) -- (319,287) -- (319,307) -- (299,307) -- cycle ;
\draw   (239,267) -- (259,267) -- (259,287) -- (239,287) -- cycle ;
\draw   (259,267) -- (279,267) -- (279,287) -- (259,287) -- cycle ;
\draw   (279,267) -- (299,267) -- (299,287) -- (279,287) -- cycle ;
\draw   (259,247) -- (279,247) -- (279,267) -- (259,267) -- cycle ;
\draw   (339,287) -- (359,287) -- (359,307) -- (339,307) -- cycle ;
\draw   (359,287) -- (379,287) -- (379,307) -- (359,307) -- cycle ;
\draw   (379,287) -- (399,287) -- (399,307) -- (379,307) -- cycle ;
\draw   (399,287) -- (419,287) -- (419,307) -- (399,307) -- cycle ;
\draw   (419,287) -- (439,287) -- (439,307) -- (419,307) -- cycle ;
\draw   (359,267) -- (379,267) -- (379,287) -- (359,287) -- cycle ;
\draw   (379,267) -- (399,267) -- (399,287) -- (379,287) -- cycle ;
\draw   (379,247) -- (399,247) -- (399,267) -- (379,267) -- cycle ;
\draw   (459,287) -- (479,287) -- (479,307) -- (459,307) -- cycle ;
\draw   (479,287) -- (499,287) -- (499,307) -- (479,307) -- cycle ;
\draw   (499,287) -- (519,287) -- (519,307) -- (499,307) -- cycle ;
\draw   (519,287) -- (539,287) -- (539,307) -- (519,307) -- cycle ;
\draw   (539,287) -- (559,287) -- (559,307) -- (539,307) -- cycle ;
\draw   (479,267) -- (499,267) -- (499,287) -- (479,287) -- cycle ;
\draw   (499,267) -- (519,267) -- (519,287) -- (499,287) -- cycle ;
\draw   (519,267) -- (539,267) -- (539,287) -- (519,287) -- cycle ;
\draw   (499,247) -- (519,247) -- (519,267) -- (499,267) -- cycle ;
\draw    (400,127) -- (400,147) ;
\draw    (400,147) -- (420,147) ;
\draw  [color={rgb, 255:red, 208; green, 2; blue, 27 }  ,draw opacity=1 ] (305,273.5) .. controls (309.08,276.83) and (312.98,280) .. (317.5,280) .. controls (322.02,280) and (325.92,276.83) .. (330,273.5) .. controls (334.08,270.17) and (337.98,267) .. (342.5,267) .. controls (346.28,267) and (349.62,269.21) .. (353,271.88) ;
\draw  [color={rgb, 255:red, 208; green, 2; blue, 27 }  ,draw opacity=1 ] (304,158.5) .. controls (308.08,161.83) and (311.98,165) .. (316.5,165) .. controls (321.02,165) and (324.92,161.83) .. (329,158.5) .. controls (333.08,155.17) and (336.98,152) .. (341.5,152) .. controls (345.28,152) and (348.62,154.21) .. (352,156.88) ;
\draw  [color={rgb, 255:red, 208; green, 2; blue, 27 }  ,draw opacity=1 ] (270,194.5) .. controls (266.67,198.58) and (263.5,202.48) .. (263.5,207) .. controls (263.5,211.52) and (266.67,215.42) .. (270,219.5) .. controls (273.33,223.58) and (276.5,227.48) .. (276.5,232) .. controls (276.5,235.78) and (274.29,239.12) .. (271.62,242.5) ;
\draw  [color={rgb, 255:red, 208; green, 2; blue, 27 }  ,draw opacity=1 ] (390,193.5) .. controls (386.67,197.58) and (383.5,201.48) .. (383.5,206) .. controls (383.5,210.52) and (386.67,214.42) .. (390,218.5) .. controls (393.33,222.58) and (396.5,226.48) .. (396.5,231) .. controls (396.5,234.78) and (394.29,238.12) .. (391.62,241.5) ;
\draw  [color={rgb, 255:red, 208; green, 2; blue, 27 }  ,draw opacity=1 ] (510,195.5) .. controls (506.67,199.58) and (503.5,203.48) .. (503.5,208) .. controls (503.5,212.52) and (506.67,216.42) .. (510,220.5) .. controls (513.33,224.58) and (516.5,228.48) .. (516.5,233) .. controls (516.5,236.78) and (514.29,240.12) .. (511.62,243.5) ;
\draw  [color={rgb, 255:red, 208; green, 2; blue, 27 }  ,draw opacity=1 ] (148,194.5) .. controls (144.67,198.58) and (141.5,202.48) .. (141.5,207) .. controls (141.5,211.52) and (144.67,215.42) .. (148,219.5) .. controls (151.33,223.58) and (154.5,227.48) .. (154.5,232) .. controls (154.5,235.78) and (152.29,239.12) .. (149.62,242.5) ;
\draw    (108,181) -- (107.02,260) ;
\draw [shift={(107,262)}, rotate = 270.71] [color={rgb, 255:red, 0; green, 0; blue, 0 }  ][line width=0.75]    (10.93,-3.29) .. controls (6.95,-1.4) and (3.31,-0.3) .. (0,0) .. controls (3.31,0.3) and (6.95,1.4) .. (10.93,3.29)   ;
\draw    (179,214) -- (244,214) ;
\draw [shift={(246,214)}, rotate = 180] [color={rgb, 255:red, 0; green, 0; blue, 0 }  ][line width=0.75]    (10.93,-3.29) .. controls (6.95,-1.4) and (3.31,-0.3) .. (0,0) .. controls (3.31,0.3) and (6.95,1.4) .. (10.93,3.29)   ;
\draw    (297,214) -- (362,214) ;
\draw [shift={(364,214)}, rotate = 180] [color={rgb, 255:red, 0; green, 0; blue, 0 }  ][line width=0.75]    (10.93,-3.29) .. controls (6.95,-1.4) and (3.31,-0.3) .. (0,0) .. controls (3.31,0.3) and (6.95,1.4) .. (10.93,3.29)   ;
\draw    (417,213) -- (482,213) ;
\draw [shift={(484,213)}, rotate = 180] [color={rgb, 255:red, 0; green, 0; blue, 0 }  ][line width=0.75]    (10.93,-3.29) .. controls (6.95,-1.4) and (3.31,-0.3) .. (0,0) .. controls (3.31,0.3) and (6.95,1.4) .. (10.93,3.29)   ;
\draw    (556,183) -- (556.97,256) ;
\draw [shift={(557,258)}, rotate = 269.24] [color={rgb, 255:red, 0; green, 0; blue, 0 }  ][line width=0.75]    (10.93,-3.29) .. controls (6.95,-1.4) and (3.31,-0.3) .. (0,0) .. controls (3.31,0.3) and (6.95,1.4) .. (10.93,3.29)   ;
\draw    (160,116) .. controls (199.6,86.3) and (454.82,86.98) .. (497.76,115.14) ;
\draw [shift={(499,116)}, rotate = 216.63] [color={rgb, 255:red, 0; green, 0; blue, 0 }  ][line width=0.75]    (10.93,-3.29) .. controls (6.95,-1.4) and (3.31,-0.3) .. (0,0) .. controls (3.31,0.3) and (6.95,1.4) .. (10.93,3.29)   ;
\draw    (160,318) .. controls (198.61,346.71) and (454.8,347) .. (497.76,318.86) ;
\draw [shift={(499,318)}, rotate = 503.37] [color={rgb, 255:red, 0; green, 0; blue, 0 }  ][line width=0.75]    (10.93,-3.29) .. controls (6.95,-1.4) and (3.31,-0.3) .. (0,0) .. controls (3.31,0.3) and (6.95,1.4) .. (10.93,3.29)   ;
\draw   (219,287) -- (239,287) -- (239,307) -- (219,307) -- cycle ;
\draw   (399,267) -- (419,267) -- (419,287) -- (399,287) -- cycle ;
\draw  [fill={rgb, 255:red, 0; green, 0; blue, 0 }  ,fill opacity=1 ] (225,297) .. controls (225,294.51) and (227.01,292.5) .. (229.5,292.5) .. controls (231.99,292.5) and (234,294.51) .. (234,297) .. controls (234,299.49) and (231.99,301.5) .. (229.5,301.5) .. controls (227.01,301.5) and (225,299.49) .. (225,297) -- cycle ;
\draw  [fill={rgb, 255:red, 0; green, 0; blue, 0 }  ,fill opacity=1 ] (246,177) .. controls (246,174.51) and (248.01,172.5) .. (250.5,172.5) .. controls (252.99,172.5) and (255,174.51) .. (255,177) .. controls (255,179.49) and (252.99,181.5) .. (250.5,181.5) .. controls (248.01,181.5) and (246,179.49) .. (246,177) -- cycle ;
\draw  [fill={rgb, 255:red, 0; green, 0; blue, 0 }  ,fill opacity=1 ] (406,137) .. controls (406,134.51) and (408.01,132.5) .. (410.5,132.5) .. controls (412.99,132.5) and (415,134.51) .. (415,137) .. controls (415,139.49) and (412.99,141.5) .. (410.5,141.5) .. controls (408.01,141.5) and (406,139.49) .. (406,137) -- cycle ;
\draw  [color={rgb, 255:red, 155; green, 155; blue, 155 }  ,draw opacity=1 ][fill={rgb, 255:red, 155; green, 155; blue, 155 }  ,fill opacity=1 ] (404.5,277) .. controls (404.5,274.51) and (406.51,272.5) .. (409,272.5) .. controls (411.49,272.5) and (413.5,274.51) .. (413.5,277) .. controls (413.5,279.49) and (411.49,281.5) .. (409,281.5) .. controls (406.51,281.5) and (404.5,279.49) .. (404.5,277) -- cycle ;
\draw  [color={rgb, 255:red, 155; green, 155; blue, 155 }  ,draw opacity=1 ][fill={rgb, 255:red, 155; green, 155; blue, 155 }  ,fill opacity=1 ] (424.5,297) .. controls (424.5,294.51) and (426.51,292.5) .. (429,292.5) .. controls (431.49,292.5) and (433.5,294.51) .. (433.5,297) .. controls (433.5,299.49) and (431.49,301.5) .. (429,301.5) .. controls (426.51,301.5) and (424.5,299.49) .. (424.5,297) -- cycle ;
\draw  [color={rgb, 255:red, 155; green, 155; blue, 155 }  ,draw opacity=1 ][fill={rgb, 255:red, 155; green, 155; blue, 155 }  ,fill opacity=1 ] (384.5,257) .. controls (384.5,254.51) and (386.51,252.5) .. (389,252.5) .. controls (391.49,252.5) and (393.5,254.51) .. (393.5,257) .. controls (393.5,259.49) and (391.49,261.5) .. (389,261.5) .. controls (386.51,261.5) and (384.5,259.49) .. (384.5,257) -- cycle ;

\draw (94,218) node   {$\varphi $};
\draw (573,219) node   {$\varphi $};
\draw (209,242) node [scale=0.7]  {$ \begin{array}{l}
remove\ \\
minimal\ \\
elements,\\
subtract\ 1
\end{array}$};
\draw (448,238) node [scale=0.7]  {$ \begin{array}{l}
\ \ \ \ add\\
maximal\\
elements
\end{array}$};
\draw (325,232) node [scale=0.7]  {$\ K\textup{- jdt}$};
\draw (320,357) node   {$\row^{-1}$};
\draw (323,79) node   {$\row^{-1}$};
\draw (150,115) node   {$T$};
\draw (272,114) node   {$T_{1}$};
\draw (390,112) node   {$T_{2}$};
\draw (511,114) node   {$T_{3}$};
\draw (151,318) node   {$T'$};
\draw (270,319) node   {$T'_{1}$};
\draw (390,321) node   {$T'_{2}$};
\draw (510,320) node   {$T'_{3}$};

\end{tikzpicture}

    \caption{Commutative diagram for proof of Theorem~\ref{thm:commute}. The red squiggles indicate weak $K$-Knuth equivalence. Note that where the large dot(s) ends up in the second to rightmost trapezoid will depend on the order ideal.}
    \label{fig:HPPWCommutes}
\end{figure}

\begin{proof}[Proof of Theorem~\ref{thm:commute} in the case of $(\mathscr R(a,b),\mathscr T(a,b))$]
All tableaux in this proof are realized as shifted tableaux. Thus the $\kjdt$ equivalence relation on tableaux and the weak $K$-Knuth equivalence relation on tableaux are the same(see \Cref{subsec:K-jdt=Knuth}).\\

Let $T$ and $T'$ be almost minimal of the rectangle and trapezoid which are $\kjdt$ equivalent (by Corollary~\ref{lem:hppwuniquematch} this is equivalent to $\phi(T)=T'$). Let $r_m=a+b$ be the rank of maximal elements in the posets. Recall
the three step definition of $K$-promotion from Definition~\ref{def:k-pro}. Let $T_1,T_2,T_3$ be the results of performing steps 1, 1 and 2 and 1,2 and 3 of $K$-promotion respectively on $T$ and define $T'_1,T'_2,T'_3$ similarly for $T'$ (thus $T_3 =\row^{-1}(T)$ and $T'_3 = \row^{-1}(T')$ by Lemma~\ref{lem:k-pro}). By Lemma \ref{lem:ResInterval}, $T|_{[2,r_m+1]}$ and $T'|_{[2,r_m+1]}$ are $\kjdt$ equivalent. Thus $T_1$ and $T'_1$ are $\kjdt$ equivalent. Performing $\kjdt$ preserves $\kjdt$ equivalence, thus $T_2$ and $T'_2$ are $\kjdt$ equivalent. By Corollary~\ref{lem:hppwuniquematch}, $T_3$ is $\kjdt$ equivalent to an almost minimal tableau $T^*$ of the trapezoid. By Lemma \ref{lem:ResInterval}, $T^*|_{[1,r_m]}$ is $\kjdt$ equivalent to $T'_2 = T'_3|_{[1,r_m]}$. By Corollary \ref{cor:intervalImpliesEqual}, $T^* = T'_3$. Thus 
\[
    \row^{-1}(T)=T_3 \stackrel{\text{$\kjdt$}}{\equiv} T^* = T'_3 = \row^{-1}(T').
\]
Finally Corollary~\ref{lem:hppwuniquematch} implies that $\phi(\row^{-1}(T)) = \row^{-1}(T')=\row^{-1}(\phi(T))$. The fact
that $\phi$ commutes with $\row^{-1}$ of course implies that it commutes with $\row$ as well.
\end{proof}

The above proof completes all cases of Theorem~\ref{thm:commute}.

\section{Remaining Conjectures on \Dpg{} Pairs}\label{conjectures}
An increasing number of results, as well as computational evidence, suggest that the posets $\calP$ and $\mathcal Q$ in a minuscule \dpg\ pair $(\calP,\mathcal Q)$ are remarkably similar. Hopkins \cite{hopkins2019minuscule} conjectured a number of properties the members of a
minuscule \dpg\ pair share that posets with isomorphic comparability graphs share, and asked why these pairs behave ``as if they have isomorphic comparability graphs'' (see \cite[Section 4]{stanley1986two} for basic properties shared by posets with isomorphic comparability graphs). As with our Theorem \ref{thm:commute}, in most of these conjectures, the only difficult case is the one involving the trapezoid poset. Here we will discuss two of Hopkins’s conjectures, and how they relate to our main theorem.

\subsection{Down-Degree Statistic and Homomesy}\label{subsec:ddeg_statistic}

For a poset $\mathcal L$, the down-degree of an element $x \in \mathcal L$ (denoted $\ddeg(x)$) is the number of elements in $\mathcal L$ which $x$ covers. Recently, there has been a great deal of interest in understanding the down-degree statistic for certain posets $\mathcal L$, and in particular computing the expected value of this statistic with respect to various natural probability distributions on $\mathcal L$ (see \cite{reinerTennerYong2018} \cite{chan2017expected} \cite{hopkins2017cde}). In the case when $\mathcal L = J(\calP)$ is the set of order ideals of another poset $\calP$, we have that $\ddeg(I) = \#\max(I)$ for $I \in J(\calP)$. In this context, the down-degree statistic is also called the antichain cardinality statistic (as explained in Section \ref{subsec:poset}, $I \mapsto \max(I)$ is a bijection between the order ideals and
antichains of $\calP$).

Hopkins \cite[Conjecture 4.9]{hopkins2019minuscule} conjectured that for any minuscule \dpg\ pair (\Cref{fig:miniscule_pairs}) \[(P,Q) \in \{(\mathscr R(a,b),\mathscr T({a,b}),(OG(6,12),H_3),(\Q^{2n},I_2(2n))\},\]
there exists a bijection $\Phi$ between rowmotion orbits of $J(\calP)$ and $J(\mathcal Q)$ such that for any rowmotion orbit $\mathcal{O} \subseteq J(\calP)$
\begin{enumerate}
    \item $\# \mathcal{O}= \# \Phi(\mathcal{O})$,
    \item $\sum_{I \in \mathcal{O}} \ddeg(I)= \sum_{I \in \Phi(\mathcal{O})} \ddeg(I)$.
\end{enumerate} Hopkins showed that such a bijection $\Phi$ exists if $\calP$ and $\mathcal Q$ are posets with isomorphic comparability graphs.

Our Theorem \ref{thm:commute} shows that we can always find a bijection $\Phi$ satisfying at least condition (1). In order to show that condition (2) is also satisfied, ideally we would show that the bijection $\varphi: J(\mathscr R(a,b)) \to J(\mathscr T(a,b))$ preserves down-degree. But this is easily seen to be false: indeed, no bijection $J(\mathscr R(a,b)) \to J(\mathscr T(a,b))$ that commutes with rowmotion can preserve down-degree for $(a,b)=(3,4)$, as there exists a rowmotion orbit $\mathcal O$ of the rectangle with different multiset $\{\ddeg(I):I \in \mathcal O\}$ compared to any rowmotion orbit $\mathcal O'$ of the trapezoid. Instead, as we now explain, to show that (2) is satisfied we can establish a homomesy result.

\begin{definition}[\cite{propp2015homomesy}]
A statistic $f$ on a finite set $S$ is said to be \emph{homomesic} with respect to an invertible operator $\Psi:S \rightarrow S$ if for all $\Psi$-orbits $\mc O$, 
\[
\frac{1}{\#\mc O} \sum_{T \in \mc O} f(T) = \frac{1}{\#S} \sum_{T \in S} f(T). 
\]
\end{definition}

Propp and Roby \cite[Theorem 27]{propp2015homomesy} explained how the Stanley-Thomas word correspondence between order ideals of the rectangle under rowmotion and binary words under rotation easily implies that down-degree is homomesic with respect to the action of rowmotion on the order ideals of the rectangle. This result was extended to all minuscule posets in \cite{rush15}. Therefore, to show that condition (2) in Hopkins’s conjecture is satisfied (for any $\Phi$ satisfying condition (1)), we only need to show that down-degree is homomesic with respect to rowmotion for each of $\mathscr T({a,b}), H_3$, and $I_2(2n)$. For $H_3, I_2(2n), \mathscr T(1,b), \mathscr T(2,b)$, and $\mathscr T(a,a)$, this is known to be true \cite[Proposition 4.13]{hopkins2019minuscule}. In what follows we will prove this homomesy also for $\mathscr T({3,b})$, and explain how our approach might possibly be extended to prove homomesy for all $\mathscr T({a,b})$.

In fact, we will prove a slightly stronger condition, namely that for certain symmetric distributions $\mu$ of order ideals in $J(\mathscr T({a,b}))$, the expected down-degree is $ab/(a+b)$.
First, we define the \emph{antichain toggleability statistics}. Let $\calP$
be a poset. For an antichain $A$ of $\calP$ and an order ideal $I \in J(\calP)$, we define
\begin{align*}
\mc T_A^+(I) :&= \begin{cases}
1 & \textnormal{if }A \not \in I \textnormal{ and }A \cup I \textnormal{ is an ideal in } J(\calP)\\
0 & \textnormal{otherwise}
\end{cases}\\
\mc T_A^-(I) :&= \begin{cases}
1 & \textnormal{if }A \in I \textnormal{ and } I \setminus A \textnormal{ is an ideal in }J(\calP) \\
0 & \textnormal{otherwise}
\end{cases}\\
\mc T_A(I) :&= \mc T_A^+(I) - \mc T_A^-(I)
\end{align*}

We commonly use $p$ to denote a single element antichain associated to the element $p$ of the poset. When $\mc T_A(I) = 1$ we say the antichain $A$ can be \textit{toggled into} the ideal $I$ and when $\mc T_A(I) = -1$ we say the antichain $A$ can be \textit{toggled out} of $I$.

The toggleability statistics $\mc T_p$ for $p\in \calP$ were considered in \cite{chan2017expected}; in particular,
in \cite{chan2017expected} the authors called a probability distribution $\mu$ on $J(\calP)$ \textit{toggle-symmetric} if
for any $p \in \calP$,
\[
    \mathbb{E}[\mu; \mc T_p(I)] = 0.
\]
Here for a probability distribution $\mu$ on a finite set $X$
and a statistic $f:X \to \mathbb{R}$ we use $\mathbb{E}[\mu;f]$ to denote the expectation of $f$ with
respect to $\mu$. We call a distribution $\mu$ on ideals in a poset \emph{toggle on anitchains-symmetric} if for any fixed antichain $A$,
\[
    \mathbb{E}[\mu; \mc T_A(I)] = 0.
\]
Notice a toggle on antichains-symmetric distribution is a toggle-symmetric distribution. A distribution which is uniform on a rowmotion
orbit $\mathcal{O} \subseteq J(\calP)$ (and zero outside this orbit) is toggle-symmetric (see
\cite[Theorem 2.14]{chan2017expected}). In fact, the same reasoning implies such a distribution is
toggle on antichains-symmetric:
\begin{lemma}\label{lem:rowissymmetric}
   A distribution $\mu$ which is uniform on a rowmotion orbit $\mathcal{O}
\subseteq J(\calP)$ is toggle on antichains-symmetric. 
\end{lemma}
\begin{proof}
    This follows immediately from the fact that $\mc T_A^+(I) = 1$ if and only if \\$\mc T_A^-(\row(I)) =1$.
\end{proof}

In \cite{chan2017expected} it was shown that for any toggle-symmetric distribution $\mu$ on $J(\mathscr R(a,b))$,
we have $\mathbb{E}[\mu;\ddeg] = (ab)/(a+b)$ (this combined with Lemma~\ref{lem:rowissymmetric} shows down-degree is homomesic with respect to the action of rowmotion on $J(\mathscr R(a,b))$.) In fact, in
\cite{chan2017expected} they proved that for a larger family of ordinary shapes $\lambda/\mu$, including the rectangle,
the expected down-degree is the same for any toggle-symmetric distribution on the order
ideals of $\lambda/\mu$; and in \cite{hopkins2017cde}, this result was extended to a family of shifted
shapes. However, it is \textit{not} the case that every toggle-symmetric distribution on
$J(\mathscr T(a,b))$ has the same expected down-degree (see \cite[Example 4.7]{hopkins2017cde}).
Nevertheless, we conjecture something slightly weaker is true:

\begin{conjecture}\label{conj:antichains}
For any toggle on antichains-symmetric distribution $\mu$ on $J(\mathscr T(a,b))$, $\mathbb{E}[\mu;\ddeg] = ab/(a+b)$.
\end{conjecture}

Observe that Conjecture~\ref{conj:antichains} together with Lemma~\ref{lem:rowissymmetric} would imply that down-degree is
homomesic with respect to the action of rowmotion on $J(\mathscr T(a,b))$ (and hence resolve the
conjecture of Hopkins mentioned at the beginning of this section).  Conjecture~\ref{conj:antichains} is
known to be true for $a=1,2,b$ (see \cite{hopkins2017cde}). We show that it also holds for $a=3$; that is,
our main result in this section is:
\begin{theorem}
\label{thm:error0antichainsym}
    For any toggle on anitchains-symmetric distribution $\mu$ on $J(\mathscr{T}(a,b))$ with $a \leq 3$
    \[
    \mathbb{E}[\mu; \ddeg] = \frac{ab}{a+b}.
    \]
\end{theorem}
\begin{corollary}
For $a \leq 3$, down-degree is homomesic with respect to the action of rowmotion
on $J(\mathscr T(a,b))$.
\end{corollary}

Before we proceed, recall the labeling of the trapezoid poset which is induced from the labeling of $\{(a,b) \in \mathbb{N}^2| a \leq b\}$, namely the minimal element is labeled $(1,1)$ and the maximal elements are labeled $(i,a+b+1-i)$
for $0 \leq i < a$, see Figure \ref{fig:rooks} for an example.  Note that as compared to the normal Cartesian plane, in our coordinate system the horizontal coordinate is the second coordinate.  For notational convenience, when dealing with the trapezoid poset $\mathscr T(a,b)$, we use $\lambda_i$ as shorthand for $a+b+1-i$.

\begin{definition}
    A \emph{rook} on the $(i,j)$ square of a trapezoid $\mathscr{T}(a,b)$ is a linear combination of statistics
    \[
    R_{i,j}:J(\mathscr{T}(a,b)) \rightarrow \mathbb{R}
    \]
    \[
    R_{i,j}(a \cdot I) = a \cdot \left( \sum_p \left(c_p^- \mc T_p^-(I) + c_p^+ \mc T_p^+(I)\right) + \sum_{i=1}^{a-1} c_{\{(i',\lambda_{i'}),(i'+1,\lambda_{i'+1})\}}^- T_{\{(i',\lambda_{i'}),(i'+1,\lambda_{i'+1})\}}^-(I) \right)
    \]
    where 
\begin{align*}
    c_{(i',j')}^- &= \begin{cases}
    1 & \textnormal{if } i' \geq i \textnormal{ and } j' \geq j\\
    -1 & \textnormal{if } i' < i \textnormal{ and } j' < j \textnormal{ and } j' > i'\\
    0 & \textnormal{otherwise}
    \end{cases}\\
    c_{(i',j')}^+ &= \begin{cases}
1 & \textnormal{if } i' \leq i \textnormal{ and } j' \leq j\\
    -1 & \textnormal{if } i' > i \textnormal{ and } j' > j \textnormal{ and } j' > i'\\
    0 & \textnormal{otherwise}
    \end{cases}\\
    c_{\{(i',\lambda_{i'}),(i'+1,\lambda_{i'+1}\}}^- &= \begin{cases}
    -1 & \textnormal{if } i' \geq i \textnormal{ and } \lambda_{i'+1} \geq j\\
    0 & \textnormal{otherwise}\\
    \end{cases}.
\end{align*}
\end{definition}

We say that the rook \emph{attacks} a square $p$ if $c_p^-+c_p^+ \ne 0$, in which case we will say the rook \emph{attacks} $p$, $c_p^-+c_p^+$ times. As we see in Figure \ref{fig:rooks}, a rook on the trapezoid will attack the squares that lie in the same row and column, with the exception that once we reach the leftmost or upmost part of the row/column, the rook starts attacking squares (or pairs of squares) that lie on the diagonal.

\begin{remark}
    The name ``rook'' comes from \cite{chan2017expected} where they used similar equations on various tableaux. On rectangular tableaux, the rook attacked all the squares in the same row and column as it; hence the name ``rook.''
\end{remark}

Figure~\ref{fig:rooks} gives a visualization of two rooks.

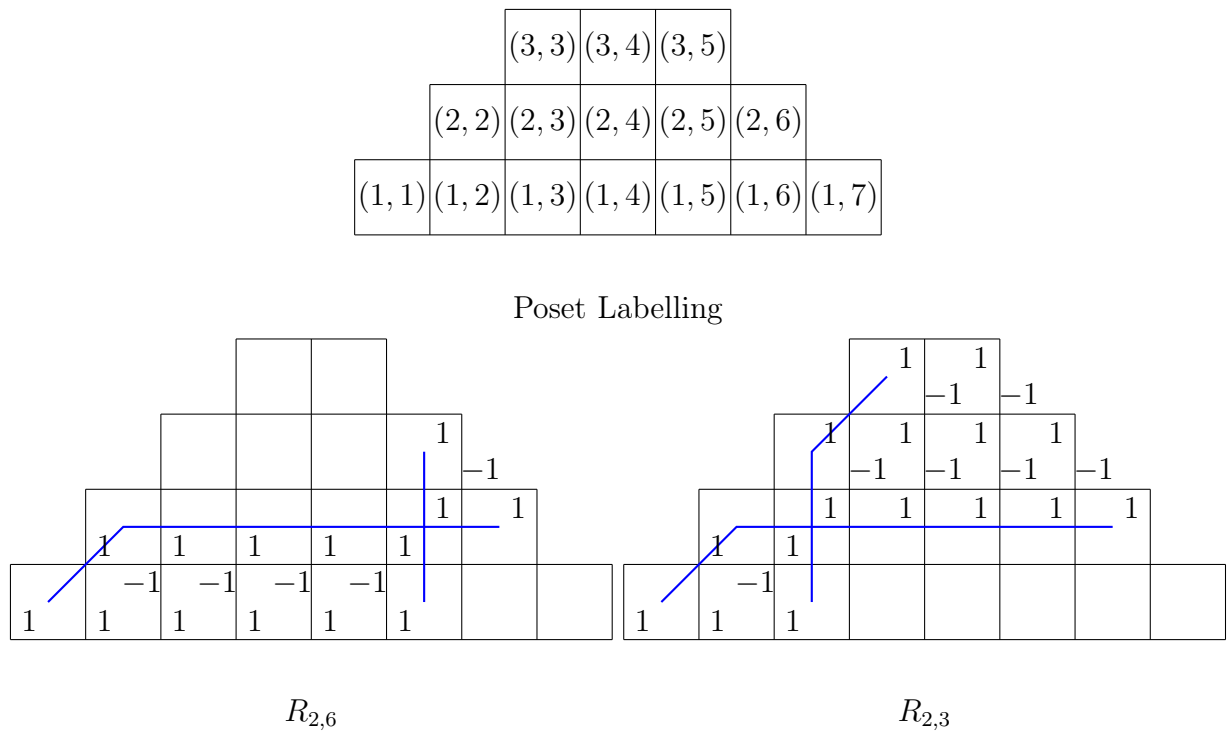
\begin{figure}
    \centering
\begin{tikzpicture}
    \draw [-] (5,2)--(2,2);
    \draw [-] (6,1)--(1,1);
    \draw [-] (7,0)--(0,0);
    \draw [-] (7,-1)--(0,-1);
    \draw [-] (5,-1)--(5,2);
    \draw [-] (6,-1)--(6,1);
    \draw [-] (1,-1)--(1,1);
    \draw [-] (0,-1)--(0,0);
    \draw [-] (7,-1)--(7,0);
    \foreach \a in {0,1,2}{
    \draw [-] (4-\a,-1)--(4-\a,2);
    }
    \foreach \i in {1,...,7}{
    \node () at (\i-0.5,-0.5) {$(1,\i)$};
    } 
    
    \foreach \i in {2,...,6}{
    \node () at (\i-0.5,0.5) {$(2,\i)$};
    } 
    
    \foreach \i in {3,...,5}
    {
    \node () at (\i-0.5,1.5) {$(3,\i)$};
    } 
    \node () at (3.5,-2) {Poset Labelling};
\end{tikzpicture}

\begin{tikzpicture}
 \draw [-] (5,3)--(3,3);
    \draw [-] (6,2)--(2,2);
    \draw [-] (7,1)--(1,1);
    \draw [-] (8,0)--(0,0);
    \draw [-] (8,-1)--(0,-1);
    \draw [-] (1,-1)--(1,1);
    \draw [-] (0,-1)--(0,0);
    \draw [-] (7,-1)--(7,0);
    \foreach \a in {-2,-1,0,1,2}{
    \draw [-] (4-\a,-1)--(4-\a,2);
    }
    \draw [-] (7,0)--(7,1);
    \draw [-] (8,-1)--(8,0);
    \foreach \a in {3,4,5}{
    \draw[-] (\a,2)--(\a,3);}

 \node () at (4,-2) {$R_{2,6}$};
 \draw [color=blue,thick](0.5,-0.5)--(1.5,0.5)--(6.5,0.5);
 \draw [color=blue,thick](5.5,-0.5)--(5.5,1.5);
 \foreach \i in {0,...,5}{
 \node () at (\i+0.25,-1+0.25){$1$};
 }
  \foreach \i in {1,...,5}{
 \node () at (\i+0.25,0+0.25){$1$};
 }
 \foreach \i in {2,...,5}{
 \node () at (\i-0.25,0-0.25){$-1$};
 }
 \node () at (6-0.25,1-0.25){$1$};
 \node () at (6+0.25,1+0.25){$-1$};
 \node () at (6-0.25,2-0.25){$1$};
 \node () at (7-0.25,1-0.25){$1$};

\end{tikzpicture}
\begin{tikzpicture}
 \draw [-] (5,3)--(3,3);
    \draw [-] (6,2)--(2,2);
    \draw [-] (7,1)--(1,1);
    \draw [-] (8,0)--(0,0);
    \draw [-] (8,-1)--(0,-1);
    \draw [-] (1,-1)--(1,1);
    \draw [-] (0,-1)--(0,0);
    \draw [-] (7,-1)--(7,0);
    \foreach \a in {-2,-1,0,1,2}{
    \draw [-] (4-\a,-1)--(4-\a,2);
    }
    \draw [-] (7,0)--(7,1);
    \draw [-] (8,-1)--(8,0);
    \foreach \a in {3,4,5}{
    \draw[-] (\a,2)--(\a,3);}

 \node () at (4,-2) {$R_{2,3}$};
 \draw [color=blue,thick](0.5,-0.5)--(1.5,0.5)--(6.5,0.5);
 \draw [color=blue,thick](2.5,-0.5)--(2.5,1.5)--(3.5,2.5);
\foreach \i in {0,1,2}{
 \node () at (\i+0.25,-1+0.25){$1$};
 }
 \foreach \i in {1,2}{
 \node () at (\i+0.25,0.25){$1$};
 }
 \foreach \i in {3,...,7}{
 \node () at (\i-0.25,1-0.25) {$1$};}
 
 \foreach \i in {3,...,6}{
 \node () at (\i-0.25,2-0.25) {$1$};}
  \foreach \i in {4,5}{
 \node () at (\i-0.25,3-0.25) {$1$};}

 \foreach \i in {3,...,6}{
 \node () at (\i+0.25,1+0.25) {$-1$};}
 
  \foreach \i in {4,5}{
 \node () at (\i+0.25,2+0.25) {$-1$};}
 
 \node () at (2-0.25,0-0.25) {$-1$};
\end{tikzpicture}

    \caption{The top pitcure depicts our labeling of $\mathscr{T}(3,5)$. The bottom picture depicts the rook statistics. In the bottom pictures, the number in the upper-right corner of a box $p$ is $c_p^{-}$ and the number in the lower-left corner is $c_p^+$.  The number in the $\llcorner$-nook between two maximal elements is $c_A^{-}$ for the antichain consisting of the two maximal elements. Where there are no numbers, the corresponding constant in the rook equation is $0$.  The blue lines mark the squares which the rook attacks.}
    \label{fig:rooks}
\end{figure}

\begin{remark}
Hopkins actually considered a slightly different rook equation which he called $R_{i,j}^{\text{shift}}$. $R_{i,j}^{\text{shift}}$ is a linear combination of toggles of single element antichains and is equal to out $R_{i,j}$ minus the term
\[
\sum_{i=1}^{a-1} c_{\{(i',\lambda_{i'}),(i'+1,\lambda_{i'+1})\}}^- T_{\{(i',\lambda_{i'}),(i'+1,\lambda_{i'+1})\}}^-(I).
\]
He called the above term $C_{i,j}^{\text{shift}}(\lambda,I)$ where $\lambda$ is the shape of the shifted Young diagram.  See \cite{hopkins2017cde} for more details.
\end{remark}
Our rooks satisfy the nice property that:

\begin{prop}\cite[Lemma 4.4]{hopkins2017cde}
For any order ideal $I \in J(\mathscr{T}(a,b))$, we have $R_{i,j}(I)=1$.
\end{prop}

Using the rook approach, we are able to give a formula for the expected down-degree.

\begin{lemma}
\label{lem:rookerrorterm}
    For any toggle-symmetric distribution $\mu$ on $J(\mathscr{T}(a,b))$,
    \[
    \mathbb{E}[\mu; \ddeg] = \frac{ab}{a+b} + \frac{a-b}{a+b}\left(\sum_{i=1}^a (a-i) \mathbb{E}[\mu;\mc T_{(i,i)}^+] - \sum_{i=1}^{a-1} i\mathbb{E}[\mu;\mc T_{(i,\lambda_{i}),(i+1,\lambda_{i+1})}^-]\right)
    \]
\end{lemma}
\begin{proof}
Consider the linear combination of statistics $f:J(\mathscr{T}(a,b)) \rightarrow \mathbb{R}$
\[
    f := (2a-a^2)R_{(1,1)} + (b-a)\sum_{j=2}^{a}R_{(1,j)} + b\sum_{i=2}^{a}R_{(i,i)} + a\sum_{j=a+1}^{a+b-1}R_{(1,j)}.
\]
We will compute $\mathbb{E}[\mu;f]$ in two ways. On one hand, for any ideal $I$,
\[
    f(I) = (2a-a^2)R_{(1,1)}(I) + (b-a)\sum_{j=2}^{a}R_{(1,j)}(I) + b\sum_{i=2}^{a}R_{(i,i)}(I) + a\sum_{j=a+1}^{a+b-1}R_{(1,j)}(I) = ab
\]
thus $\mathbb{E}[\mu;f] = ab$.
On the other hand, we see that the rook arrangement which $f$ defines attacks every element of $\mathscr{T}(a,b)$ $a+b$ times except for elements of form $(i,i)$, which are attacked $(a+b)-(a-b)\cdot (a-i)$ times. Thus for some constants $c_p$, 
\[
    f = (a+b) \cdot \ddeg - \sum_{i=1}^a (a-b)(a-i)\mc T_{(i,i)}^+ + \sum_{i=1}^{a-1} (a-b)i\mc T_{(i,\lambda_{i}),(i+1,\lambda_{i+1})} + \sum_p c_p \mc T_p
\]
Over a toggle-symmetric distribution $\mu$,    $\mathbb{E}[\mu;\sum_p c_p \mc T_p] = 0$.
Using linearity of expectation:
\[
\mathbb{E}\left[\mu;\sum_p c_p \mc T_p\right] = (a+b) \cdot \mathbb{E}[\mu;\ddeg] - \sum_{i=1}^a (a-b)(a-i)\mathbb{E}\left[\mu;\mc T_{(i,i)}^+\right] + \sum_{i=1}^{a-1} (a-b)i\mathbb{E}\left[\mu;\mc T_{(i,\lambda_{i}),(i+1,\lambda_{i+1})}^-\right].
\] 
Comparing our two equations for $\mathbb{E}[\mu,f]$ yields the desired result.
\end{proof}

Our goal is now to evaluate the ``error term'', that is, the term \[\frac{a-b}{a+b}\left(\sum_{i=1}^a (a-i) \mathbb{E}[\mu;\mc T_{(i,i)}^+] - \sum_{i=1}^{a-1} i\mathbb{E}[\mu;\mc T_{(i,\lambda_{i}),(i+1,\lambda_{i+1})}^-]\right)\]
appearing in Lemma~\ref{lem:rookerrorterm}. We will show that the error term is $0$ for toggle
on antichains-symmetric distributions $\mu$ on $\mathscr{T}(a,b)$ for $a\leq 3$.

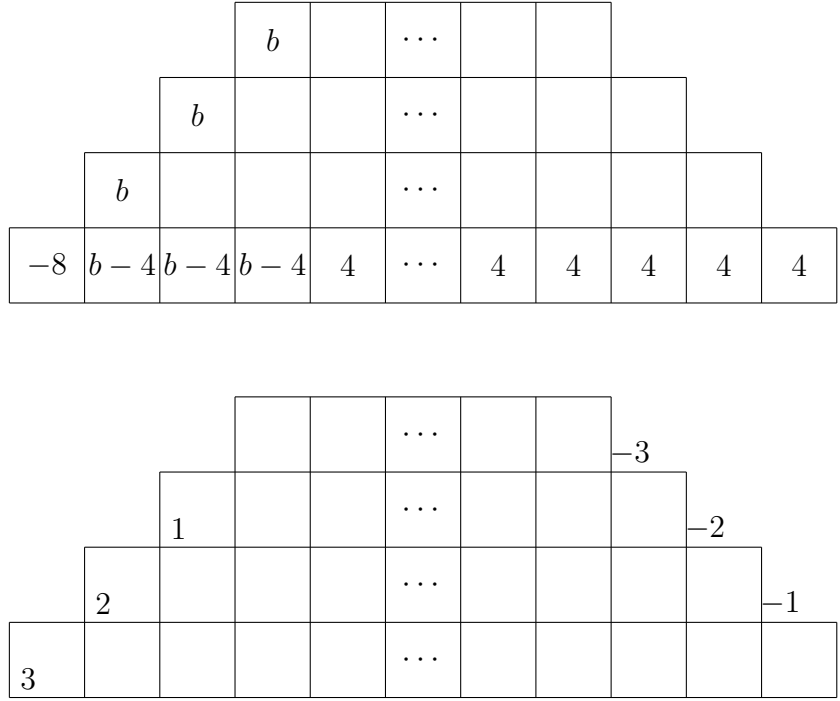
\begin{figure}
    \centering
   \begin{tikzpicture}
 \draw [-] (8,3)--(3,3);
    \draw [-] (9,2)--(2,2);
    \draw [-] (10,1)--(1,1);
    \draw [-] (11,0)--(0,0);
    \draw [-] (11,-1)--(0,-1);
    \foreach \i in {0,...,11}{
    \draw (\i,0)--(\i,-1);
    }
    \foreach \i in {1,...,10}{
    \draw (\i,0)--(\i,1);
    }
    \foreach \i in {2,...,9}{
    \draw (\i,2)--(\i,1);
    }
    \foreach \a in {3,...,8}{
    \draw[-] (\a,2)--(\a,3);}
    \node at (0.5,-0.5) {$-8$};
    \foreach \a in {1,2,3}{
     \node at (0.5+\a,-0.5+\a) {$b$};}
     
      \foreach \a in {1,2,3}{
     \node at (0.5+\a,-0.5) {$b-4$};}
     \node at (4.5,-0.5) {$4$};
     \foreach \a in {0,1,2,3}{
    \node at (5.5,-0.5+\a) {$\cdots$}; }
    \foreach \a in {2,...,6}{
    \node at (4.5+\a,-0.5) {$4$};}
\end{tikzpicture}

\vspace{1.2cm}

\begin{tikzpicture}
 \draw [-] (8,3)--(3,3);
    \draw [-] (9,2)--(2,2);
    \draw [-] (10,1)--(1,1);
    \draw [-] (11,0)--(0,0);
    \draw [-] (11,-1)--(0,-1);
    \foreach \i in {0,...,11}{
    \draw (\i,0)--(\i,-1);}
    \foreach \i in {1,...,10}{
    \draw (\i,0)--(\i,1);}
    \foreach \i in {2,...,9}{
    \draw (\i,2)--(\i,1);}
    \foreach \a in {3,...,8}{
    \draw[-] (\a,2)--(\a,3);}
 \foreach \a [evaluate=\a as \suivant using int(4-\a)] in {1,2,3}{
     \node at (3.25-\a,2.25-\a) {$\a$};}
     \foreach \a in {0,1,2,3}{
    \node at (5.5,-0.5+\a) {$\cdots$}; }
    \foreach \a [evaluate=\a as \suivant using int(4-\a)] in {1,2,3}{
     \node at (11.25-\a,-0.75+\a) {$-\a$};}
\end{tikzpicture}

    \caption{Rook arrangement defined by $f$ for $\mathscr T(4,b)$ and error term scaled down by $(a+b)/(a-b)$.}
    \label{fig:my_label}
\end{figure}

We will show the error term is $0$ for toggle on anitchains-symmetric distribution $\mu$ on $\mathscr T(a,b)$ with $a \leq 3$:

\begin{lemma}
\label{lem:errorterm2is0}
For any toggle-symmetric distribution $\mu$ on $J(\mathscr{T}(a,b))$
\[
\sum_{i=1}^{a} (a-i) \mathbb{E}[\mu;\mc T_{(i,i)}^-] - \sum_{(i,j) | j< i \leq a} \mathbb{E}[\mu;\mc T_{(i,j)}^-] = 0.
\]
\end{lemma}
\begin{proof}
For any event $Y$ and a probability distribution $\mu$ on a finite set $X$, we use $\Pr[\mu;Y]$ to denote the probablity $Y$ occurs with respect to $\mu$.
Define
\begin{align*}
    a_n &:= \sum_{i = 1}^{n-1} \mathbb{E}[\mu;\mc T_{(i,n)}^-] \\
    &= \Pr[\mu; I \textnormal{ has a maximal element in } \{(i,n)| i<n\}] \\
    b_n &:=\mathbb{E}[\mu;\mc T_{(n,n)}^-]\\
    &= \Pr[\mu; (n,n) \textnormal{ is a maximal element in } I]
\end{align*}
Importantly, notice that 
\begin{align*}
    a_n + b_n &= \Pr[\mu; I \textnormal{ has a maximal element in } \{(i,n)\}] \\
    &= \Pr[\mu; I^C \textnormal{ has a minimal element in } \{(i,n+1)| i<n+1\}]\\
    &= \Pr[\mu; I \textnormal{ has a maximal element in } \{(i,n+1)| i<n+1\}] \tag{\textnormal{by toggle-symmetry}}\\
    &= a_{n+1}
\end{align*}
where $I^C$ is the complement of the ideal $I$. The second to last equality follows from the fact that if an ideal $I$ is cut out with a $\urcorner$-shaped nook identifying a maximal element of $I$ in the $j$-th collumn, then it must be followed by a $\llcorner$-shaped nook identifying a minimal element of the complement of $I$ in the $j+1$-th collumn and vice versa. Noting that $a_1=0$, we now compute:
\begin{align*}
    \sum_{i=1}^{a} (a-i) \mathbb{E}[\mu;\mc T_{(i,i)}^-] - \sum_{(i,j) | j< i \leq a} \mathbb{E}[\mu;\mc T_{(i,j)}^-]
    &= \sum_{i=1}^{a} (a-i) b_i - \sum_{i=1}^{a} a_i \\
    &= \left( \sum_{i=1}^{a} (a-i) (a_{i+1}-a_i)\right) - \sum_{i=1}^{a} a_i \\ 
    &= \sum_{i=1}^{a} a_i - \sum_{i=1}^{a} a_i = 0.
\end{align*}
\end{proof}

In particular, the above lemma says that a toggle-symmetric distribution $\mu$ satisfies $\mathbb{E}[\mu;\ddeg] = ab/(a+b)$ if and only if our error term form Lemma~\ref{lem:rookerrorterm} is equal to
\[
\sum_{i=1}^{a} (a-i) \mathbb{E}[\mu;\mc T_{(i,i)}^-] - \sum_{(i,j) | j< i \leq a} \mathbb{E}[\mu;\mc T_{(i,j)}^-] = 0.
\]
Thus we have 
\begin{prop}
\label{prop:errorterm2suffices}
A toggle-symmetric distribution $\mu$ on $J(\mathscr{T}(a,b))$ satisfies $\mathbb{E}[\mu;\ddeg] = ab/(a+b)$ if and only if
\[
\sum_{(i,j) | j< i \leq a} \mathbb{E}[\mu;\mc T_{(i,j)}^-] = \sum_{i=1}^{a-1} i\mathbb{E}[\mu;\mc T^-_{(i,\lambda_i),(i+1,\lambda_{i+1})}].
\]
\end{prop}

\begin{lemma}
\label{lem:2elmadj}
For a toggle on antichains-symmetric distribution $\mu$ on
$J(\mathscr T(a,b))$,
    \begin{align*}
    \sum_{i=1}^{a-1} \mathbb{E}[\mu; \mc T^-_{(i,\lambda_i),(i+1,\lambda_{i+1})}]
    = \sum_{i=1}^{a-1} \sum_{j>i} \mathbb{E}[\mu;\mc T^-_{(i,j+1),(j,j)}] = \sum_{i=1}^{a-1} \mathbb{E}[\mu; \mc T^-_{(i,i+1)}].
    \end{align*}
\end{lemma}
\begin{proof}
Define the subsets of 
antichains
\begin{align*}
S^- :&= \{\{(i,j),(i',j') \} \subseteq \mathscr{T}(a,b) |j = j'+1, i < i' \leq j' , j<\lambda_i\},\\
S^+ :&= \{\{(i,j),(i',j') \} \subseteq \mathscr{T}(a,b) | j = j'+1, i < i' < j' , j \leq \lambda_i\}\\
S :&= S^+ \cap S^-.
\end{align*}
To get the first equality, we consider the following linear combination of statistics $f: J(\mathscr{T}(a,b)) \rightarrow \mathbb{R}$
\begin{align*}
    f :&= \sum_{A \in S^+} \mc T_A^+ - \sum_{A \in S^-} \mc T_A^- \\
        &= \sum_{A \in S} \mc T_A + \sum_{i=1}^{a-1}  \mc T^+_{(i,\lambda_i),(i+1,\lambda_{i+1})} - \sum_{i=1}^{a-1} \sum_{j< i} \mu;T^-_{(i,i),(j,i+1)}
\end{align*}
We compute $\mathbb{E}[\mu;f]$ in two ways. First notice that if you can toggle in an antichain in $S^+$, then you can toggle out an antichain in $S^-$, thus for all ideals $I$, $f(I) = 0$ and so $\mathbb{E}[\mu;f] = 0$. Next by toggle on anitchains-symmetry and linearity of expectations:
\begin{align*}
\mathbb{E}[\mu;f] &= \sum_{A \in S} \mathbb{E}[\mu;\mc T_A] + \sum_{i=1}^{a-1} \mathbb{E}[\mu; \mc T^+_{(i,\lambda_i),(i+1,\lambda_{i+1})}] - \sum_{i=1}^{a-1} \sum_{j<i} \mathbb{E}[\mu;T^-_{(i,i),(j,i+1)}] \\
&= \sum_{i=1}^{a-1} \mathbb{E}[\mu; \mc T^+_{(i,\lambda_i),(i+1,\lambda_{i+1})}] - \sum_{i=1}^{a-1} \sum_{j<i} \mathbb{E}[\mu;T^-_{(i,i),(j,i+1)}]
\end{align*}
We get the second equality from toggle on antichains-symmetry:
\begin{align*}
    \sum_{i=1}^{a-1} \mathbb{E}[\mu; \mc T^-_{(i,i+1)}] &=  \sum_{i=1}^{a-1} \sum_{j<i} \mathbb{E}[\mu;T^+_{(i,i),(j,i+1)}] \\
    &= \sum_{i=1}^{a-1} \sum_{j<i} \mathbb{E}[\mu;T^-_{(i,i),(j,i+1)}].
\end{align*}
\end{proof}

\begin{proof}[Proof of Theorem~\ref{thm:error0antichainsym}]
$ $\newline
\textbf{Case 0:} $a=1$:\\ In this case, our trapezoid poset is isomorphic to the rectangle poset $\mathscr R_{1,b}$.\\
\textbf{Case 1:} $a=2$:\\
This follows from Proposition~\ref{prop:errorterm2suffices} and Lemma~\ref{lem:2elmadj}.\\
\textbf{Case 2:} $a=3$\\
Between Proposition~\ref{prop:errorterm2suffices} and Lemma~\ref{lem:2elmadj}, all that remains, is to show
\[
\mathbb{E}[\mu; \mc T^-_{(1,3)}] = \mathbb{E}[\mu; \mc T^-_{(1,\lambda_1),(2,\lambda_2)}].
\]
Define $S$ to be the set of anitchains
\[
    S:= \left\{A \subseteq \mathscr{T}(3,n) | 
                \begin{array}{ll}
                  A=\{(3,j),(2,j+1)(1,j+2) \} \textnormal{ where } 3 \leq j \leq \lambda_3 - 1\}\\
                  \textnormal{or } A=\{(3,j),(2,j+1)(1,j+3)\}  \textnormal{ where } 3 \leq j \leq \lambda_3 - 1\}
                \end{array}
              \right\}.
\]
Then define a linear combination of statistics $f: J(\mathscr{T}(a,b)) \rightarrow \mathbb{R}$
\[
    f:= \sum_{A \in S} \mc T_A - \mc T^-_{(1,3)} + \mc T_{(2,3)(1,4)} + \mc T_{(2,2)(1,4)} + \mc T_{(2,3)(1,5)}
    + \mc T^+_{(1,\lambda_1),(2,\lambda_2)}.
\]
We compute $\mathbb{E}[\mu;f]$ in two ways. First we show that for any ideal $I$, $f(I)=0$. For any ideal $I$ where least one of 
\[\mc T^-_{(1,3)}(I),\; \mc T_{(2,3)(1,4)}(I),\; \mc T_{(2,2)(1,4)}(I),\; \mc T_{(2,3)(1,5)}(I),\; \mc T^+_{(1,\lambda_1),(2,\lambda_2)}(I)
\]
is nonzero, it can be checked by hand that exactly $1$ antichain in our sum for $f$ can be toggled out of $I$ and exactly $1$ antichain in our our sum for $f$ can be toggled into $I$, yielding $f(I) = 0$. For any other ideal, if an antichain $\{(3,x),(2,y),(1,z)\}$ of $S$ can be toggled out of $I$, then $\{(3,x+1),(2,y+1),(1,z+1)\} \in S$ can be toggled into $I$. Similarly if an antichain $\{(3,x),(2,y),(1,z)\}$ of $S$ can be toggled into $I$, then $\{(3,x-1),(2,y-1),(1,z-1)\} \in S$ can be toggled out of $I$. Thus $f(I) = 0$.
On the other hand, from toggle on antichains-symmetry,
\[
    \mathbb{E}[\mu;f] = \mathbb{E}[\mu; \mc T^-_{(1,3)}] - \mathbb{E}[\mu; \mc T^-_{(1,\lambda_1),(2,\lambda_2)}].
\]
Thus we conclude the theorem.
\end{proof}

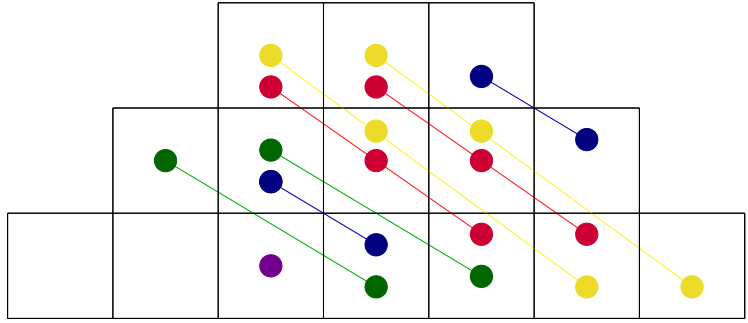
\begin{figure}[h!]
    \centering
    \begin{tikzpicture}[scale=1.4]
    \draw [-] (5,2)--(2,2);
    \draw [-] (6,1)--(1,1);
    \draw [-] (7,0)--(0,0);
    \draw [-] (7,-1)--(0,-1);
    \draw [-] (5,-1)--(5,2);
    \draw [-] (6,-1)--(6,1);
    \draw [-] (1,-1)--(1,1);
    \draw [-] (0,-1)--(0,0);
    \draw [-] (7,-1)--(7,0);
    \foreach \a in {0,1,2}{
    \draw [-] (4-\a,-1)--(4-\a,2);
    }

   \draw [color=green!70!black](1.5,0.5)--(3.5,-0.7);
   \filldraw [color=green!40!black](1.5,0.5) circle (3pt);
   \filldraw [color=green!40!black](3.5,-0.7) circle (3pt);
   
   \draw [color=green!70!black](2.5,0.6)--(4.5,-0.6);
   \filldraw [color=green!40!black](2.5,0.6) circle (3pt);
   \filldraw [color=green!40!black](4.5,-0.6) circle (3pt);
   
   \draw [color=blue!70!black](2.5,0.3)--(3.5,-0.3);
   \filldraw [color=blue!50!black](2.5,0.3) circle (3pt);
   \filldraw [color=blue!50!black](3.5,-0.3) circle (3pt);

   \filldraw [color=blue!50!black](2.5,0.3) circle (3pt);
 
   \filldraw [color=purple!60!blue] (2.5,-0.5) circle (3pt);
   
   \draw [color=blue!70!black](4.5,1.3)--(5.5,0.7);
   \filldraw [color=blue!50!black](4.5,1.3) circle (3pt);
   \filldraw [color=blue!50!black](5.5,0.7) circle (3pt);
   
   \draw [color=red!90] (2.5,1.2)--(4.5,-0.2);
   \filldraw [color=red!80!blue] (2.5,1.2) circle (3pt);
   \filldraw [color=red!80!blue] (4.5,-0.2) circle (3pt);
   \filldraw [color=red!80!blue] (3.5,0.5) circle (3pt);
   
   \draw [color=red!90] (2.5+1,1.2)--(4.5+1,-0.2);
   \filldraw [color=red!80!blue] (2.5+1,1.2) circle (3pt);
   \filldraw [color=red!80!blue] (4.5+1,-0.2) circle (3pt);
   \filldraw [color=red!80!blue] (3.5+1,0.5) circle (3pt);
   
   \draw [color=yellow!90] (2.5,1.5)--(5.5,-0.7);
   \filldraw [color=yellow!90!black] (2.5,1.5) circle (3pt);
   \filldraw [color=yellow!90!black] (3.5,0.78) circle (3pt);
   \filldraw [color=yellow!90!black] (5.5,-0.7) circle (3pt);
   
   \draw [color=yellow!90] (2.5+1,1.5)--(5.5+1,-0.7);
   \filldraw [color=yellow!90!black] (2.5+1,1.5) circle (3pt);
   \filldraw [color=yellow!90!black] (3.5+1,0.78) circle (3pt);
   \filldraw [color=yellow!90!black] (5.5+1,-0.7) circle (3pt);
   
   \draw [-] (5,2)--(2,2);
    \draw [-] (6,1)--(1,1);
    \draw [-] (7,0)--(0,0);
    \draw [-] (7,-1)--(0,-1);
    \draw [-] (5,-1)--(5,2);
    \draw [-] (6,-1)--(6,1);
    \draw [-] (1,-1)--(1,1);
    \draw [-] (0,-1)--(0,0);
    \draw [-] (7,-1)--(7,0);
    \foreach \a in {0,1,2}{
    \draw [-] (4-\a,-1)--(4-\a,2);
    }
\end{tikzpicture}
    \caption{The antichains involved in the proof of Theorem~\ref{thm:error0antichainsym} for $\mathscr T(3,b)$}
\end{figure}

\subsection{Piecewise-Linear Rowmotion}\label{subsec:PLrow}
There are a few different generalizations of rowmotion which act on $\calP$-partitions rather than order ideals. One such generalization is to view a $\calP$-partition in a poset as an ideal in the poset cross a chain and then do rowmotion to the corresponding ideal.  Unfortunately, the orbit structures of rowmotion on $J(\mathscr{R}(a,b) \times [\ell])$ and $J(\mathscr{T}(a,b) \times [\ell])$ are not always the same.

Another generalization of rowmotion is \emph{piecewise-linear rowmotion}, as introduced in \cite{pwbr_row}. For a poset $\calP$, the order polytope $\mathcal{O}(P)$ is the polytope in $\mathbb{R}^\calP$ defined by
\begin{enumerate}
    \item $0 \leq f(x) \leq 1$ for all $x \in \calP$,
    \item $f(x) \leq f(y)$ for all $x \leq y$.
\end{enumerate}

There is a bijection
$(1/\ell)\mathbb{Z}^{\calP} \cap \mathcal{O}(\calP) \simeq \PP^{\ell}(\calP)$ which is just multiplication
by $\ell$. The \emph{piecewise-linear toggle} on a coordinate $p\in \calP$ of the order polytope $\mathcal{O}(\calP)$ is the map $\tau_p^{\PL}\colon \mathcal{O}(\calP)\to \mathcal{O}(\calP)$ defined by
\[
    \tau^{\PL}_p(f)(q) = \begin{cases}
    f(q) & q \neq p \\
    \min(\{1\} \cup \{f(x) :x>p\}) + \max(\{0\} \cup \{f(x): x<p\}) - f(p) & q = p 
    \end{cases}.
\]
Piecewise-linear rowmotion is the product of piecewise-linear toggles for any linear extension $p_1,p_2 \hdots p_n$,
\[
    \row^{PL} := \tau^{\PL}_{p_1} \circ \tau^{\PL}_{p_2} \circ \hdots \circ \tau^{\PL}_{p_n}.
\] Via the identification $(1/\ell)\mathbb{Z}^{\calP} \cap \mathcal{O}(\calP) \simeq \PP^{\ell}(\calP)$, we get an action \[\row^{\PL}: \PP^{\ell}(\calP) \to \PP^{\ell}(\calP)\] of piecewise-linear rowmotion on $\calP$-partitions.

Hopkins \cite{hopkins2019minuscule} conjectures that the orbit structures of $\calP$-partitions of height $\ell$ on minuscule \dpg\ pairs under piecewise-linear rowmotion are the same.
Our main theorem answers this conjecture in the case $\ell = 1$. However, the bijection $\varphi$ fails to commute with piecewise-linear rowmotion on higher height $\calP$-partitions.

A question this raises is whether a slight modification of $\varphi$ can commute with piecewise-linear rowmotion. There is a natural way to modify $\varphi$ to yield a bijection between $\IT^{\ell}(\mathscr R(a,b))$ and $\SIT^{\ell}(\mathscr T(a,b))$ (and hence via Remark~\ref{rem:PPtoTableaux} between $\calP$-partitions of these posets):

\begin{prop}
    In Definition \ref{def:hppw_bijection}, replacing the minimal tableaux $S$ with any other unique rectification target $S'$ yields a bijection 
    \[
    \varphi_{S'}(T) := \kjdt_{S'^{-1}(\overline{1})} \circ \kjdt_{S'^{-1}(\overline{2})} \hdots \kjdt_{S'^{-1}(\overline{\max(S')})}(T')
    \]
    between $\IT^{\ell}(\mathscr R(a,b))$ and $\SIT^{\ell}(\mathscr T(a,b))$.
\end{prop}
\begin{proof}
    As shown in the $\varphi$ bijection, the minimal tableaux of the rectangle $M_{\mathscr R(a,b)}$ rectifies to the minimal tableaux of the trapezoid $M_{\mathscr T(a,b)}$. Since the latter is a unique rectification target and $\varphi_{S'}(M_{\mathscr R(a,b)})$ is a straight shifted tableaux, 
    \[
    \varphi_{S'}(M_{\mathscr R(a,b)}) = M_{\mathscr T(a,b)}.
    \]
    Let $S''$ be the skew shifted tableau which $S'$ transforms into under the $\kjdt$ slides in $\varphi_{S'}$ i.e.
    \[
        S'' := \widehat{\kjdt}_{M_{\mathscr R(a,b)}^{-1}(a+b-1)} \circ \widehat{\kjdt}_{M_{\mathscr R(a,b)}^{-1}(a+b-2)} \hdots \widehat{\kjdt}_{M_{\mathscr R(a,b)}^{-1}(1)}(S').
    \]
    Since $S'$ is a unique rectification target, there is an map from $\SIT^{\ell}(\mathscr T(a,b)) \rightarrow \IT^{\ell}(\mathscr R(a,b))$ given by
    \[
        T \mapsto \widehat{\kjdt}_{S''^{-1}(a+b-1)} \circ \widehat{\kjdt}_{S''^{-1}(a+b-2)} \hdots \widehat{\kjdt}_{S''^{-1}(1)}(T).
    \]
    Since every $\swap$ is invertible, this map is an injection.  Since $\#\SIT^{\ell}(\mathscr T(a,b)) = \#\IT^{\ell}(\mathscr R(a,b))$, this map is a bijection.  Applying the \textit{infusion involution} of Thomas and Yong \cite{TY09} tells us the inverse of this bijection is $\varphi_{S'}$.
\end{proof}

Theorem \ref{shapeunique} shows that $\varphi_{S'}$ and $\varphi$ act the same on almost minimal tableaux. Thus $\varphi_{S'}$ also commutes with rowmotion on order ideals. Unfortunately for $\mathscr{R}(3,n)$ every $\varphi_{S'}$ acts the same on all $\calP$-partitions. Thus even the modified version of $\varphi$ fails to commute with piecewise-linear rowmotion for heights $>1$.

\bibliographystyle{alpha}
\bibliography{references}

\end{document}